\documentclass[11pt]{article}
\usepackage[titletoc,title]{appendix}
\usepackage{amsmath,amsfonts,amssymb,bm,amsthm}
\usepackage[colorlinks,citecolor=blue,linkcolor=blue,pagebackref]{hyperref}
\usepackage{tikz}
\usepackage{esint}
\usetikzlibrary{calc,decorations.markings}
\usetikzlibrary{shapes,arrows}
\usetikzlibrary{intersections,shapes.arrows}
\usetikzlibrary{decorations.pathreplacing}
\usetikzlibrary{arrows.meta}
\numberwithin{equation}{section}
\newtheorem{theorem}{Theorem}[section]
\newtheorem{lemma}[theorem]{Lemma}
\newtheorem{proposition}[theorem]{Proposition}
\newtheorem{corollary}[theorem]{Corollary}

\theoremstyle{definition}

\newtheorem{definition}[theorem]{Definition}

\theoremstyle{remark}
\newtheorem{remark}[theorem]{Remark}

\usepackage{multicol}
\usepackage{longtable}

\usepackage[T1]{fontenc}
\usepackage[top=1in, bottom=1in, left=1in, right=1in]{geometry}
\usepackage{fancyhdr}
\newcommand{\dr}{\mathrm{d}}
\newcommand{\dist}{\operatorname{dist}}
\newcommand{\Ric}{\operatorname{Ric}}
\newcommand{\Riem}{\operatorname{Riem}}
\newcommand{\Sc}{\operatorname{Sc}}
\newcommand{\curl}{\operatorname{curl}}
\newcommand{\prin}{\mathrm{prin}}
\newcommand{\sub}{\mathrm{sub}}
\newcommand{\loc}{\mathrm{loc}}
\usepackage{enumerate}
\usepackage{xcolor,cancel}
\usepackage{relsize}%for resizing mathematical symbols

\renewcommand{\tilde}{\widetilde}

\usepackage{orcidlink}

\allowdisplaybreaks

\usepackage{pifont}
\renewcommand*{\backrefalt}[4]{%
\ifcase #1 %
No citations%
\or
\ding{43}~p.~#2%
\else
\ding{43}~pp.~#2%
\fi}

\begin{document}

\title{A microlocal pathway to spectral asymmetry:\\ curl and the eta invariant}
\author{Matteo Capoferri\,\orcidlink{0000-0001-6226-1407}\thanks{MC:
Dipartimento di Matematica ``Federigo Enriques'', 
Università degli Studi di Milano, 
Via C.~Saldini 50, 20133 Milan, Italy;
Maxwell Institute for Mathematical Sciences, 
Edinburgh
\&
Department of Mathematics,
Heriot-Watt University,
Edinburgh EH14 4AS, UK;
%m.capoferri@hw.ac.uk,
matteo.capoferri@unimi.it
\url{https://mcapoferri.com}.
}
\and
Dmitri Vassiliev\,\orcidlink{0000-0001-5150-9083}\thanks{DV:
Department of Mathematics,
University College London,
Gower Street,
London WC1E~6BT,
UK;
D.Vassiliev@ucl.ac.uk,
\url{http://www.homepages.ucl.ac.uk/\~ucahdva/}.
}}

%%%% this removes footnote "crosses" by the authors' names
\renewcommand\footnotemark{}
%%%%

%\date{version 76; 10 December 2025; \LaTeX{ed} \today}
\date{10 December 2025}

\maketitle

\vspace{-.6cm}

\begin{abstract}
The notion of eta invariant is traditionally defined by means of analytic continuation. We prove, by examining the particular case of the operator curl, that the eta invariant can equi\-valently be obtained as the trace of the difference of positive and negative spectral projections, appropriately regularised. Our construction is direct, in the sense that it does not involve analytic continuation, and is based on the use of pseudodifferential techniques. This provides a novel approach to the study of spectral asymmetry of non-semibounded (pseudo)differential systems on manifolds which encompasses and extends previous results.

\


\

{\bf 2020 MSC classes: }
primary
58J50; %Spectral problems; spectral geometry; scattering theory on manifolds
secondary
35P20, %Asymptotic distributions of eigenvalues in context of PDEs
35Q61, %Maxwell equations
47B93, %Operators arising in mathematical physics
47F99; %Partial differential operators -- None of the above, but in this section
58J28, %Eta-invariants
58J40. %Pseudodifferential and Fourier integral operators on manifolds

\end{abstract}

\tableofcontents

\allowdisplaybreaks

\section{Statement of the problem}
\label{Statement of the problem}

The study of \emph{spectral asymmetry}, namely, the difference in the distribution of positive and negative eigenvalues of (pseudo)differential operators on manifolds is a well-established area of pure mathematics, initiated by Atiyah, Patodi and Singer in their seminal series of papers \cite{asymm1, asymm2, asymm3, asymm4}. The classical approach to the subject goes as follows. Given a non-semibounded self-adjoint first order (pseudo)differential operator with (real) eigenvalues $\{\lambda_k\}$, one introduces the quantity
\[
\eta(s):=\sum_{\lambda_k\neq 0} \frac{\operatorname{sgn}(\lambda_k)}{|\lambda_k|^s}, \qquad s\in \mathbb{C},
\]
called the \emph{eta function of the operator}. After showing that the eta function can be defined as a meromorphic function in the whole complex plane with no pole at $s=0$, one declares the value $\eta(0)$, called \emph{eta invariant} of the operator, to be a measure of the spectral asymmetry of the operator. The motivation underpinning this definition is that for a self-adjoint operator acting in a finite-dimensional vector space the quantity $\eta(0)$ is precisely the number of positive eigenvalues minus the number of negative eigenvalues. Let us mention that one can also define a \emph{local} version of the eta function, where each term in the infinite series is weighted by the modulus squared of the corresponding eigenfuncions, and for which similar results can be proved --- see, e.g., \cite{gilkey_local,bismut}.

The eta invariant, a geometric invariant, turned out to be a powerful concept with far-reaching consequences across analysis, geometry, and beyond. For instance, it features in the Atiyah--Singer Index Theorem for elliptic operators on manifolds with boundary, to name but one of its most well-known applications. The existing body of work on the topic is massive, especially in the case of Dirac and Dirac-type operators, hence we do not even attempt a systematic review of the literature, which would go beyond the scope of the current paper. We shall just emphasise that most of the existing approaches hinge on a combination of techniques from complex analysis (analytic continuation) and differential topology (characteristic classes), often relying on black-box-type arguments.

This paper completes the analysis initiated in \cite{curl},
resulting in a new, direct approach to the study of eta functions (both local and global) and eta invariants through the prism of microlocal analysis, starting from a less well-studied, yet fundamental, operator: the operator curl.
We refer the reader to \cite{curl} for a more detailed review of existing literature.

\

Let $(M,g)$ be a connected closed oriented Riemannian manifold of dimension $d=3$. We denote by $\rho(x):=\sqrt{\det g_{\alpha\beta}(x)}$ the Riemannian density and by $\Omega^k=\Omega^k(M)$, $0\le k\le3$, the space of real-valued $k$-forms over $M$. Furthermore, we denote by $*$, $\dr$ and $\delta$ the Hodge dual, the exterior derivative (differential) and the codifferential, respectively. Finally, we denote by $\Riem$, $\Ric$ and $\Sc$ the Riemann tensor, the Ricci tensor and scalar curvature\footnote{The Riemann curvature tensor $\Riem$ has components ${\Riem^\kappa}_{\lambda\mu\nu}$ defined in accordance with 
\[
{\Riem^\kappa}_{\lambda\mu\nu}:=
\dr x^\kappa(\Riem(\partial_\mu\,,\partial_\nu)\,\partial_\lambda)
=
\partial_\mu{\Gamma^\kappa}_{\nu\lambda}
-\partial_\nu{\Gamma^\kappa}_{\mu\lambda}
+{\Gamma^\kappa}_{\mu\eta}{\Gamma^\eta}_{\nu\lambda}
-{\Gamma^\kappa}_{\nu\eta}{\Gamma^\eta}_{\mu\lambda}\,,
\]
the $\Gamma$'s being Christoffel symbols. The Ricci tensor is defined as $\Ric_{\mu\nu}:=R^\alpha{}_{\mu\alpha\nu}$ and $\Sc:=g^{\mu\nu}\Ric_{\mu\nu}$ is scalar curvature.}. We refer the reader to \cite[Appendix~A]{curl} for our differential geometric conventions.

We equip $\Omega^k(M)$ with the $L^2$ inner product
\begin{equation}
\label{inner product}
\langle u,v \rangle:=\int_M *u \wedge v=\int_M u \wedge * v\,,
\end{equation}
where $\wedge$ is the exterior product of differential forms, and define $H^s(M)$, $s>0$, to be the space of differential forms that are square integrable together with their partial derivatives up to order $s$. We do not carry in our notation for Sobolev spaces the degree of differential forms: this will be clear from the context. Henceforth, to further simplify notation we drop the $M$ and write $\Omega^k$ for $\Omega^k(M)$ and $H^s$ for $H^s(M)$.

\

Hodge's Theorem \cite[Corollary~3.4.2]{jost} tells us that $\Omega^k$ decomposes into a direct sum of three orthogonal closed subspaces
\begin{equation*}
\label{Hodge decompostion}
\Omega^k
=
\dr\Omega^{k-1}
\oplus
\delta\Omega^{k+1}
\oplus
\mathcal{H}^k,
\end{equation*}
where $\dr\Omega^{k-1}$, $\delta\Omega^{k+1}$ and $\mathcal{H}^k$ are the Hilbert subspaces of exact, coexact and harmonic $k$-forms, respectively.

\

\begin{definition}
\label{definition curl}
We define $\curl$ to be the operator
\begin{equation}
\label{definition curl equation}
\operatorname{curl}=*\dr : \delta\Omega^2 \cap H^1\to \delta\Omega^2\,.
\end{equation}
\end{definition}

Observe that Definition~\ref{definition curl} makes sense, because $*\dr$ maps coexact 1-forms to coexact 1-forms. It is well-known --- see, e.g., \cite{giga,baer_curl} and \cite[Theorem~2.1]{curl} --- that $\curl$ as defined by \eqref{definition curl equation} is a self-adjoint operator with discrete spectrum accumulating to both $+\infty$ and $-\infty$. Furthermore, zero is not an eigenvalue of $\curl$.
Note, however, that $\curl$ is not elliptic\footnote{Recall that, by definition, a matrix (pseudo)differential operator is elliptic if the determinant of its principal symbol is nonvanishing on $T^*M\setminus\{0\}$.}. Indeed,  the formula for the principal symbol of $\operatorname{curl}$ (which happens to coincide with its full symbol) reads
\begin{equation}
\label{principal symbol curl}
[\operatorname{curl}_\mathrm{prin}]_\alpha{}^\beta(x,\xi)= -i\,E_\alpha{}^{\beta\gamma}(x)\,\xi_\gamma\,,
\end{equation}
where the tensor $E$ is defined in accordance with 
\begin{equation}
\label{Levi-Civita tensor}
E_{\alpha\beta\gamma}(x):=\rho(x)\,\varepsilon_{\alpha\beta\gamma}
\end{equation}
and $\varepsilon$ is the totally antisymmetric symbol,
$\varepsilon_{123}:=+1$.
%Although $\xi$ is not a 1-form in its own right, at a formal level the RHS of \eqref{principal symbol curl} can be rewritten, concisely, as $-i*\xi$. 
A straightforward calculation shows that
\begin{equation*}
\label{determinant principal symbol curl}
\det (\operatorname{curl}_\mathrm{prin})=0.
\end{equation*}
The eigenvalues of $\operatorname{curl}_\mathrm{prin}$ are simple and read
\begin{equation*}
\label{eigenvalues principal symbol curl}
h^{(0)}(x,\xi)=0, \qquad h^{(\pm)}(x,\xi)=\pm\|\xi\|\,,
\end{equation*}
for all $(x,\xi)\in T^*M\setminus \{0\}$, where
\begin{equation*}
\label{norm of xi}
\|\xi\|
%:=(g^{\mu\nu}(x)\,\xi_\mu\xi_\nu)^{1/2}
:=\sqrt{g^{\mu\nu}(x)\,\xi_\mu\xi_\nu}\,.
\end{equation*}
Consequently,  standard elliptic theory does not apply and particular care is required when studying the spectrum of $\curl$.

\

Let $\lambda_j$ be the eigenvalues of $\curl$ and $u_j$ its orthonormalised eigenforms. Here we enumerate using positive integers $j$ for positive eigenvalues and negative integers $j$ for negative eigenvalues, so that
\begin{equation*}
\label{6 November 2022 equation 1}
-\infty \leftarrow\ldots\le\lambda_{-2}\le\lambda_{-1}<0<\lambda_1\le\lambda_2\le\ldots \rightarrow +\infty\,,
\end{equation*}
with account of multiplicities.

\begin{definition}
\label{definition eta functions}
We define the \emph{local eta function} for the operator $\curl$ as
\begin{equation}
\label{local eta function for curl}
\eta_{\curl}^\loc(x;s)
:=
\sum_{j\in\mathbb{Z}\setminus\{0\}}\frac{\operatorname{sgn}\lambda_j}{|\lambda_j|^s}\,*(u_j\wedge *u_j)(x)\,,
\end{equation}
where
$*(u_j\wedge *u_j)(x)=g^{\alpha\beta}(x)\,[u_j]_\alpha(x)\,[u_j]_\beta(x)$ 
is the pointwise norm squared of the eigenform $u_j$.

We define the \emph{(global) eta function} for the operator $\curl$ as
\begin{equation}
\label{global eta function for curl}
\eta_{\operatorname{curl}}(s)
:=
\sum_{k\in\mathbb{Z}\setminus\{0\}}\frac{\operatorname{sgn}\lambda_k}{|\lambda_k|^s}\,.
\end{equation}
\end{definition}

It is not hard to show that the series \eqref{global eta function for curl} converges absolutely for $\operatorname{Re} s>3$ and defines a meromorphic function in $\mathbb{C}$ by analytic continuation, with possible first order poles at $s=3,2,1,0,-1,-2,\dots$. Furthermore, it is holomorphic at $s=0$, see \cite{asymm1,asymm2,asymm3,asymm4}. Similarly, for any fixed $x\in M$, the series \eqref{local eta function for curl} converges absolutely for $\operatorname{Re} s>3$ and defines a meromorphic function in $\mathbb{C}$ by analytic continuation, with possible first order poles at $s=3,2,1,-1,-2,\dots$. It is shown in \cite{5terms} that, in fact, $s\mapsto \eta_{\curl}(s)$ and, for each $x\in M$, $s\mapsto \eta_{\curl}^\loc(x;s)$ are holomorphic in the half-plane $\operatorname{Re}s>-1$; in fact, this also follows from the arguments presented in Section~\ref{Proof of Theorem main result}.

Of course, the local and global eta functions are related via the identity
\begin{equation*}
\label{relation between local and global eta functions}
\eta_{\operatorname{curl}}(s)
=
\int_M
\eta_{\operatorname{curl}}^\loc(x;s)\,\rho(x)\,\dr x\,.
\end{equation*}

\begin{definition}
\label{definition local and global eta invariants}
We call the real-valued function $\eta_{\curl}^\loc(x;0)$ the \emph{local eta invariant} for the operator $\curl$, and we call the scalar $\eta_{\curl}(0)$ the \emph{eta invariant} for the operator $\curl$.
\end{definition}

The overall philosophy of our paper is to avoid the ``black box'' of analytic continuation and approach the eta invariant intuitively, namely, 
\begin{equation*}
\eta_{\curl}(0)=\#\{\text{positive eigenvalues}\}
\,-\,
\#\{\text{negative eigenvalues}\}\,.
\end{equation*}
Indeed, for a self-adjoint operator in a finite-dimensional inner product space the quantity $\eta(0)$ can be written as
\begin{equation}
\label{eta invariant in finite dimensional setting}
\operatorname{Tr}(P_+-P_-)
\end{equation}
where $\operatorname{Tr}$ is the operator trace, and $P_+$ and $P_-$ are projections onto the direct sums of positive and negative eigenspaces, respectively.

Trying to give a meaning to \eqref{eta invariant in finite dimensional setting} for the case of non-semibounded operators acting in infinite-dimensional Hilbert spaces is both the key and the starting point of our approach, whose foundations were developed in \cite{part1,part2,diagonalisation,curl} and are summarised below, for the reader's convenience.

\

The issue at hand is how to rigorously define
\eqref{eta invariant in finite dimensional setting}
when the projection operators $P_+$ and $P_-$ have infinite rank, and understand the meaning of the outcome. 

Let us go back to the operator $\curl$ and let us introduce the following orthogonal projections acting in $\Omega^1$:
\begin{equation}
\label{definition Pplus}
P_+:=\sum_{j=1}^{+\infty}u_j \langle u_j, \,\cdot\, \rangle\,,
\end{equation}
\begin{equation}
\label{definition Pminus}
P_-:=\sum_{j=1}^{+\infty}u_{-j} \langle u_{-j}, \,\cdot\, \rangle\,,
\end{equation}
\begin{equation}
\label{definition P0}
P_0:=-\dr \Delta^{-1}\delta\,,
\end{equation}
where $\Delta^{-1}$ is the pseudoinverse of the (nonpositive) Laplace--Beltrami operator $\Delta:=-\delta \dr$. The operators $P_+$ and $P_-$ are the positive and negative spectral projections, whereas $P_0$ is the orthogonal projection onto exact 1-forms.

The operators \eqref{definition Pplus}--\eqref{definition P0} are related as
\begin{equation*}
P_++P_-+P_0=\operatorname{Id}-P_{\mathcal{H}^1},
\end{equation*}
where $\operatorname{Id}$ is the identity operator and $P_{\mathcal{H}}$ is the orthogonal projection onto the (finite-dimensional) subspace of harmonic 1-forms. It was shown in \cite{curl} that \eqref{definition Pplus}--\eqref{definition P0} are pseudodifferential operators of order zero, whose full symbol can be explicitly constructed via the algorithm given in \cite[Section~4.3]{part1}.

In what follows, we denote by $\Psi^s$ the space of classical pseudodifferential operators of order $s$ with polyhomogeneous symbols acting on 1-forms. Furthermore, we define
\begin{equation}
\label{definition psi minus infinity}
\Psi^{-\infty}:=\bigcap_s \Psi^s
\end{equation}
and we write $Q=R \mod \Psi^{-\infty}$ if $Q-R$ is an integral operator with infinitely smooth integral kernel. Recall that a pseudodifferential operator $Q\in \Psi^s$ acting on 1-forms can be written locally as
\begin{equation}
\label{11 June 2021 equation 1}
Q: u_\alpha(x)
\mapsto
v_\alpha(x)
=(2\pi)^{-d}
\int
e^{i(x-y)^\gamma\xi_\gamma}\,q_\alpha{}^\beta(x,\xi)\,
u_\beta(y)\,\dr y\,\dr\xi\,.
\end{equation}
The quantity
\begin{equation}
\label{11 June 2021 equation 2}
q_\alpha{}^\beta(x,\xi)
\sim
[q_s]_\alpha{}^\beta(x,\xi)
+
[q_{s-1}]_\alpha{}^\beta(x,\xi)
+\dots, 
\end{equation}
is called the \emph{(full) symbol} of $Q$. Here $\,\sim\,$ stands for asymptotic expansion \cite[\S~3.3]{shubin}. The components $[q_{s-k}]_\alpha{}^\beta$ of $q_\alpha{}^\beta$ are positively homogeneous in momentum $\xi$:
\begin{equation*}
[q_{s-k}]_\alpha{}^\beta(x, \lambda \,\xi)=\lambda^{s-k}\,[q_{s-k}]_\alpha{}^\beta(x, \xi), \qquad \forall\lambda>0, \quad k=0,1,2,\dots.
\end{equation*} 
Observe that the indices $\alpha$ and $\beta$ in \eqref{11 June 2021 equation 2} `live' at different points, $x$ and $y$ respectively. When writing \eqref{11 June 2021 equation 1} we implicitly used the same coordinate system for $x$ and $y$.

The leading homogeneous component of the symbol --- a smooth matrix function on $T^*M\setminus \{0\}$ --- is called the \emph{principal symbol} of $Q$ and denoted by $Q_\prin$. We denote by $Q_\sub$ the subprincipal symbol of $Q$ --- a  modification of the subleading homogeneous component of the symbol of $Q$ --- defined in accordance with \cite[Definition~3.2]{curl}. Observe that principal and subprincipal symbols are covariant quantities under changes of local coordinates, see~\cite[Remark~3.3]{curl}.

Throughout the paper we denote pseudodifferential operators with upper case letters, and their symbols and homogeneous components of symbols with lower case. At times we will write a pseudo\-differential operator as an integral operator with distributional integral kernel (Schwartz kernel); namely, we will write \eqref{11 June 2021 equation 1} as
\begin{equation}
\label{Q acting on 1-forms 1}
Q: u_\alpha(x)\mapsto\int_M \mathfrak{q}_\alpha{}^\beta(x,y)\,u_\beta(y)\,\rho(y)\,\dr y
\end{equation}
with
\begin{equation*}
\label{frak q in terms of q}
\mathfrak{q}_\alpha{}^\beta(x,y)
=
\frac{1}{(2\pi)^d\,\rho(y)}
\int
e^{i(x-y)^\gamma\xi_\gamma}\,q_\alpha{}^\beta(x,\xi)
\,\dr\xi
\end{equation*}
in a distributional (local) sense and modulo an infinitely smooth contribution.
When we do so, we use lower case Fraktur font for the Schwartz kernel.

\

The key notion which allows us to tackle \eqref{eta invariant in finite dimensional setting} is that of (pointwise) matrix trace of a pseudo\-differential operator
acting on 1-forms.

\begin{definition}
\label{definition matrix trace introduction}
Let $Q\in \Psi^s$. We call the \emph{matrix trace} of $Q$ the scalar pseudodifferential operator $\operatorname{\mathfrak{tr}} Q$ of order $s$ defined as
\begin{equation}
\label{definition matrix trace introduction equation}
\operatorname{\mathfrak{tr}} Q: f(x)\mapsto \int_M (\operatorname{\mathfrak{tr}}\mathfrak{q})(x,y)\, f(y)\, \rho(y)\, \dr y\,,
\end{equation}
where
\begin{equation}
\label{Q acting on 1-forms 6}
(\operatorname{\mathfrak{tr}}\mathfrak{q})(x,y)
:=
\mathfrak{q}_{\alpha}{}^\beta(x,y)
\,
Z_\beta{}^\alpha(y,x)
\,\chi(\operatorname{dist}(x,y)/\epsilon)
\end{equation}
is the pointwise matrix trace of the distributional kernel $\mathfrak{q}$ of $Q$.
\end{definition}

In \eqref{definition matrix trace introduction equation} $Z_{\alpha}{}^\beta(x,y)$ is the linear map \eqref{Q acting on 1-forms 3} realising parallel transport of vectors from $x$ to $y$ along the unique shortest geodesic connecting them, $\chi:[0,+\infty) \to \mathbb{R}$ is a compactly supported smooth scalar function such that $\chi=1$ in a neighbourhood of zero, $\dist$ is the geodesic distance, and $\epsilon>0$ is a sufficiently small parameter ensuring that \eqref{Q acting on 1-forms 6} vanishes when $x$ and $y$ are far away. We refer the reader to Section~\ref{The matrix trace of an operator, revisited} for a more detailed description of these quantities, as well as a more thorough discussion of the properties of the matrix trace of a (pseudo)differential operator.

Note that in Definition~\ref{definition matrix trace introduction}
we write $\operatorname{\mathfrak{tr}}$ with Fraktur font to emphasise
the fact that \eqref{Q acting on 1-forms 6} is not the standard matrix trace
$\mathrm{tr}$, but it involves parallel transport.

\

The essential idea underpinning Definition~\ref{definition matrix trace introduction} is to decompose the operation of taking the (operator) trace of an operator acting on 1-forms which is, \emph{a priori}, not necessarily of trace class, into two separate steps: first one takes the matrix trace of the original operator, thus obtaining a scalar pseudodifferential operator, then one takes the operator trace of the resulting scalar operator, in the hope that cancellations along the way would make the latter operation legitimate. 

Remarkably, this turns out to be the case for the operator $P_+-P_-$ associated with $\curl$.

\begin{definition}
\label{definition asymmetry operator}
We define the \emph{asymmetry operator} to be the self-adjoint scalar pseudodifferential operator
\begin{equation}
\label{definition asymmetry operator equation}
A:=\operatorname{\mathfrak{tr}} (P_+-P_-)\,,
\end{equation}
where $P_\pm$ are the operators \eqref{definition Pplus}--\eqref{definition Pminus} and $\operatorname{\mathfrak{tr}}$ is the matrix trace \eqref{definition matrix trace introduction equation}.
\end{definition}

\emph{Prima facie}, $A$ is an operator of order zero. It was shown in \cite[Theorem~1.4]{curl} that $A$ is, in fact, a pseudodifferential operator of order $-3$ (hence \emph{almost} of trace class\footnote{Recall that in dimension $d=3$ a sufficient condition for a self-adjoint pseudodifferential operator to be of trace class is that its order be strictly less than $-3$.}) with principal symbol
\begin{equation}
\label{Aprin}
A_\mathrm{prin}(x,\xi)=-\frac1{2\|\xi\|^5}E^{\alpha\beta \gamma}(x)\, \nabla_\alpha \operatorname{Ric}_{\beta}{}^\rho(x)\, \xi_\gamma\xi_\rho\,.
\end{equation}

A careful analysis of the leading singularity of the integral kernel $\mathfrak{a}$ of the operator $A$ --- essentially encoded within \eqref{Aprin} --- allowed us to prove the following result in \cite{curl}.

\begin{theorem}[{\cite[Theorem 1.6]{curl}}]
\label{main theorem 3 from curl}
The integral kernel $\mathfrak{a}(x,y)$ of the asymmetry operator $A$ is a bounded function, smooth outside the diagonal. Furthermore, for any $x\in M$ the limit
\begin{equation}
\label{regularised local trace of curl}
\psi_{\operatorname{curl}}^\loc(x)
:=
\lim_{r\to0^+}
\frac{1}{4\pi r^2}
\int_{\mathbb{S}_r(x)}
\mathfrak{a}(x,y)\,\dr S_y
\end{equation}
exists and defines a continuous scalar function $\,\psi_{\operatorname{curl}}^\loc:M\to\mathbb{R}\,$. Here $\mathbb{S}_r(x)=\{y\in M|\dist(x,y)=r\}$ is the sphere of radius $r$ centred at $x$
and $\dr S_y$ is the surface area element on this sphere.
\end{theorem}

\begin{definition}
\label{definition regularised local and global trace}
We call the function $\psi_{\operatorname{curl}}^\loc: M \to \mathbb{R}$ the \emph{regularised local trace} of $\curl$ and the number
\begin{equation*}
\label{regularised global trace of curl}
\psi_{\operatorname{curl}}:=\int_M \psi_{\operatorname{curl}}^\loc(x) \,\rho(x)\,\dr x
\end{equation*}
the \emph{regularised global trace} of $\curl$.
\end{definition}
These two quantities are geometric invariants, in that they are determined by the Riemannian 3-manifold and its orientation. In particular, the global regularised trace is a real number measuring the asymmetry of the spectrum of $\curl$.

\

In \cite{curl} we claimed that $\psi_{\operatorname{curl}}$ is precisely the classical eta invariant and that $\psi_{\operatorname{curl}}^\loc$ is the local eta invariant, only defined in a purely analytic fashion, via microlocal analysis. The main goal of this paper is to prove this claim.

\begin{remark}
We should point out that there are many different approaches to the subject of spectral asymmetry existing in the literature. It is worth mentioning, for instance, the one relying on a local invariant known as Wodzicki residue \cite{wodzicki}, see also \cite{bruning,loya} for further details, later adopted, in various guises, by various authors, including Guillemin, Dixmier, and Connes.

Let us emphasise that our approach is distinct from and does not rely on the Wodzicki--Guillemin--Dixmier--Connes construction. As explained earlier in this section, our underlying idea is to take the matrix trace first, which leads to a massive regularisation reducing the order of the pseudodifferential operator by 3, thus
resulting in what we call the asymmetry operator. This idea can be deployed, in a user-friendly manner, in a variety of physically meaningful scenarios, see, e.g., \cite{dirac_asymmetry}. 

Also note that the asymmetry operator is no longer a projection/idempotent, so
 Wodzicki-type results don't immediately apply in our case. 
\end{remark}

\vspace{1cm}

\subsection*{Notation}
\addcontentsline{toc}{subsection}{Notation}

%\begin{multicols}{2}
\begin{longtable}{l l}
\hline
\\ [-1em]
\multicolumn{1}{c}{\textbf{Symbol}} & 
  \multicolumn{1}{c}{\textbf{Description}} \\ \\ [-1em]
 \hline \hline \\ [-1em]
$\sim$ & Asymptotic expansion \\ \\ [-1em]
$\ast$ & Hodge dual \\ \\ [-1em]
$\|\,\cdot\,\|$ & Riemannian norm \\ \\ [-1em]
$|\,\cdot\,|$ & Euclidean norm \\ \\ [-1em]
%$\langle\,\cdot\,\rangle$ & Japanese bracket \eqref{Japanese bracket}\\ \\ [-1em]
$A$ & Asymmetry operator, Definition~\ref{definition asymmetry operator} \\ \\ [-1em]
$A^{(s)}$ & Parameter-dependent asymmetry operator, Definition~\ref{definition parameter dependent asymmetry operator} \\ \\ [-1em]
$A^{(s)}_\mathrm{diag}$, $A^{(s)}_\mathrm{pt}$ & Local decomposition of $A^{(s)}$ as per~\eqref{As diag} and~\eqref{As pt} \\ \\ [-1em]
$\mathfrak{a}(x,y)$ & Integral kernel of the asymmetry operator $A$ \\ \\ [-1em]$\mathfrak{a}^{(s)}(x,y)$ & Integral kernel of the asymmetry operator $A$ \\ \\ [-1em]
$\curl$ & The operator curl \eqref{definition curl} \\ \\ [-1em]
%$\curl_E$ & Extended curl, Definition~\ref{definition extended curl}  \\ \\ [-1em]
%$\curl_{E,\dr},\curl_{E,\delta},\curl_{E,\mathcal{H}}$ & Orthogonal summands of $\curl_E$, Lemma~\ref{lemma invariant subspaces extended curl} \\ \\ [-1em]
$d$ & Dimension of the manifold $M$, $d\ge2$\\ \\ [-1em]
$\dr$ & Exterior derivative \\ \\ [-1em]
$\delta$ & Codifferential \\ \\ [-1em]
$\Delta:=-\delta\dr$ & (Nonpositive) Laplace--Beltrami operator \\ \\ [-1em]
$\boldsymbol{\Delta}:=-(\dr\delta+\delta\dr)$ & (Nonpositive) Hodge Laplacian\\ \\ [-1em]
$\dist$ & Geodesic distance \\ \\ [-1em]
%$e_j{}^\alpha(x)$, $e^k{}_\beta(x)$ & Framing and dual framing \eqref{dual framing} \\ \\ [-1em]
%$\tilde e_j{}^\alpha(x)$, $\tilde e^k{}_\beta(x)$ & Levi-Civita framing and dual Levi-Civita framing, Definition~\ref{definition Levi-Civita framing} \\ \\ [-1em]
$\varepsilon_{\alpha\beta\gamma}$ & Totally antisymmetric symbol, $\varepsilon_{123}=+1$ \\ \\ [-1em]
$E_{\alpha\beta\gamma}$ & Totally antisymmetric tensor \eqref{Levi-Civita tensor} \\ \\ [-1em]
$f_{x^\alpha}$ & Partial derivative of $f$ with respect to $x^\alpha$ \\ \\ [-1em]
$g$ & Riemannian metric \\ \\ [-1em]
%$G$ & Einstein tensor \\ \\ [-1em]
%$\gamma(x,y;\tau)$ & Geodesic connecting $x$ to $y$, with $\gamma(x,y;0)=x$ and $\gamma(x,y;1)=y$ \\ \\ [-1em]
$\Gamma^\alpha{}_{\beta\gamma}$ & Christoffel symbols\\ \\ [-1em]
$\eta_Q^\loc(x;s)$ & Local eta function of the operator $Q$ (formula \eqref{local eta function for curl} for $Q=\curl$) \\ \\ [-1em]
$\eta_Q(s)$ & Eta function of the operator $Q$ (formula \eqref{global eta function for curl} for $Q=\curl$) \\ \\ [-1em]
$H^s(M)$ & Generalisation of the usual Sobolev spaces $H^s$ to differential forms \\ \\ [-1em]
$\mathcal{H}^k(M)$ & Harmonic $k$-forms over $M$ \\ \\ [-1em]
$(\lambda_j, u_j)$, $j=\pm1, \pm2, \ldots$ & Eigensystem for $\curl$ \\ \\ [-1em]
$\theta$ & Heaviside theta function \\ \\ [-1em]
$\operatorname{I}$ & Identity matrix \\ \\ [-1em]
$\operatorname{Id}$ & Identity operator \\ \\ [-1em]
$(\mu_j, f_j)$, $j=0,1,2,\ldots$ & Eigensystem for $-\Delta$ \\ \\
[-1em]
$M$ & Connected closed oriented Riemannian manifold\\ \\ [-1em]
$\operatorname{mod} \ \Psi^{-\infty}$ & Modulo an integral operator with infinitely smooth kernel \\ \\ [-1em]
$P_0$, $P_\pm$ & Orthogonal projections~\eqref{definition P0} and~\eqref{definition Pplus}, \eqref{definition Pminus} \\ \\ [-1em]
$p_\pm(x,y)$ & Full symbol of $P_\pm$ \\ \\ [-1em]
$Q_\prin$ & Principal symbol of the pseudodifferential operator $Q$ \\ \\ [-1em]
%$Q_{\prin,s}$ & Principal symbol of $Q$, a pseudodifferential operator of order $-s$ \\ \\ [-1em]
$Q_{\sub}$ & Subprincipal symbol of $Q$, for operators on 1-forms see \cite[Definition~3.2]{curl}  \\ \\ [-1em]
$R_\lambda$  & Resolvent $(-\boldsymbol{\Delta}-\lambda \mathrm{I})^{-1}$ of $-\boldsymbol{\Delta}$  \\ \\ [-1em]
$\mathrm{Ref}^{(s)}$  & Reference operator, Definition~\ref{definition of the reference operator}  \\ \\ [-1em]
$\mathfrak{ref}^{(s)}(x,y)$  & Scalar function~\eqref{The reference operator equation 2} and distribution~\eqref{distribution}  \\ \\ [-1em]
$\Riem$, $\Ric$, $\Sc$ & Riemann curvature tensor, Ricci tensor, scalar curvature \\ \\ [-1em]
$\rho(x)$ & Riemannian density \\ \\ [-1em]
$\mathbb{S}_r(x)$ & Geodesic sphere of radius $r$ centred at $x\in M$\\ \\ [-1em]
$\operatorname{\mathfrak{tr}}$ & (Pointwise) matrix trace, Definition~\ref{definition matrix trace introduction} \\ \\ [-1em]
$\operatorname{Tr}$ & Operator trace  \\ \\ [-1em]
$TM$, $T^*M$ & Tangent and cotangent bundle \\ \\ [-1em]
$\Omega^k(M)$ & Differential $k$-forms over $M$ \\ \\ [-1em]
$\psi_{\curl}^\loc(x)$ & Regularised local trace of $A$, Definition~\ref{definition regularised local and global trace} \\ \\ [-1em]
$\psi_{\curl}$ & Regularised global trace of $A$, Definition~\ref{definition regularised local and global trace} \\ \\ [-1em]
$\Psi^s$ & Classical pseudodifferential operators of order $s$ \\ \\ [-1em]
$\Psi^{-\infty}$ & Infinitely smoothing operators \eqref{definition psi minus infinity} \\ \\ [-1em]
%$\zeta_Q(s)$ & Zeta function of the operator $Q$ \\ \\ [-1em]
$Z$ & Parallel transport map \eqref{Q acting on 1-forms 3}\\ \\ [-1em]
\hline
\end{longtable}
%\end{multicols}

\section{Main results}
\label{Main results}

The centrepiece of our paper is the following result.

\begin{theorem}
\label{Theorem main result}
The regularised local trace of the asymmetry operator coincides with the local eta invariant and the regularised global trace of the asymmetry operator coincides with the eta invariant. Namely,
\begin{equation*}
\label{Theorem main result equation 1}
\psi_{\operatorname{curl}}^\loc(x)
=
\eta_{\operatorname{curl}}^\loc(x;0)\,,
\end{equation*}
\begin{equation*}
\label{Theorem main result equation 2}
\psi_{\operatorname{curl}}
=
\eta_{\operatorname{curl}}(0)\,.
\end{equation*}
\end{theorem}
This completes the analysis initiated in \cite{curl} and provides a new approach to the study of spectral asymmetry of non-semibounded (pseudo)\-differential systems on manifolds which
possesses the following main elements of novelty.
\begin{itemize}
\item 
We characterise asymmetry of the spectrum in terms of a pseudodifferential operator of negative order, as opposed to a single number. The classical geometric invariants can be recovered by computing the local and global regularised operator traces of the asymmetry operator. In this sense, our approach encompasses and extends previous results.

\item
The overarching idea is not specific to $\curl$, but can be deployed for a variety of operators. For example, we will be applying it to the massless Dirac operator in a separate paper.

\item
Our construction is direct, in the sense that it does not involve analytic continuation, and is based on the use of pseudodifferential techniques and explicit computations.

\item
It implements in a rigorous fashion the intuitive understanding of spectral asymmetry as the difference between the number of positive and negative eigenvalues.

\end{itemize}

\

The strategy for proving Theorem~\ref{Theorem main result} is as follows.

\

Let us introduce a one-parameter family of pseudodifferential operators defined as follows.

\begin{definition}
\label{definition parameter dependent asymmetry operator}
Let $s$ be a real number. We define the \emph{parameter-dependent asymmetry operator} as
\begin{equation}
\label{definition parameter dependent asymmetry operator equation 1}
A^{(s)}:=\operatorname{\mathfrak{tr}}\left[(P_+-P_-) (-\boldsymbol{\Delta})^{-s/2}\right]\,,
\end{equation}
where $\boldsymbol{\Delta}$ is the (nonpositive) Hodge Laplacian acting on 1-forms.
\end{definition}

\emph{Prima facie}, for a given $s\in \mathbb{R}$ the operator \eqref{definition parameter dependent asymmetry operator equation 1} is a pseudodifferential operator of order $-s$. Furthermore, it is not hard to to see that for $s>3$ the operator $A^{(s)}$ is of trace class, and that its operator trace is the eta function $\eta_{\curl}(s)$ --- full justification for these claims will be provided in Section~\ref{The parameter-dependent asymmetry operator}.

A deeper analysis reveals that, very much like the asymmetry operator \eqref{definition asymmetry operator equation}, $A^{(s)}$ enjoys higher smoothing properties than initially expected.

\begin{theorem}
\label{Theorem order of A(s)}

\phantom{a}
\begin{enumerate}[(a)]
\item
The operator $A^{(s)}$ is a self-adjoint pseudodifferential operator of order $-s-3$.

\item
The principal symbol of the operator $A^{(s)}$ reads
\begin{equation}
\label{Theorem order of A(s) equation 1}
(A^{(s)})_\mathrm{prin}(x,\xi)
=
-\frac{(s+1)(s+3)}{6\,\|\xi\|^{s+5}}\,
E^{\alpha\beta \gamma}(x)\,
\nabla_\alpha \operatorname{Ric}_{\beta}{}^\rho(x)\, \xi_\gamma\xi_\rho
\,.
\end{equation}
\end{enumerate}
\end{theorem}

\begin{remark}
\label{remark about -1 and -3}
Formula~\eqref{Theorem order of A(s) equation 1}
tells us that the principal symbol of the operator $A^{(s)}$ vanishes at $s=-1$ and $s=-3$.
In fact, a stronger result is true: the operator $A^{(s)}$ itself vanishes at $s=-1$ and $s=-3$.
Indeed, according to formulae
\eqref{alternative representation for As}
and
\eqref{tth power of Hodge Laplacian}
we have $\,A^{(-k)}=\operatorname{\mathfrak{tr}}(\curl^k)\,$ for odd natural $k$,
whereas
Propositions~\ref{proposition trace curl} and \ref{proposition trace curl cubed}
tell us that $\,\operatorname{\mathfrak{tr}}\curl=\operatorname{\mathfrak{tr}}(\curl^3)=0\,$.
\end{remark}

Theorem~\ref{Theorem order of A(s)} part (a) allows us to extend our earlier claim about the operator $A^{(s)}$ being trace class from $s>3$ all the way up to $s>0$.

\begin{theorem}
\label{theorem trace A(s)}
Let $\mathfrak{a}^{(s)}(x,y)$ be the integral kernel of $A^{(s)}$. For $s>0$ we have
\begin{equation*}
\label{Theorem dealing with A(s) equation 1}
\mathfrak{a}^{(s)}(x,x)=\eta_{\curl}^\loc(x;s)
\end{equation*}
and
\begin{equation*}
\label{Theorem dealing with A(s) equation 2}
\operatorname{Tr}A^{(s)}=\eta_{\operatorname{curl}}(s)\,.
\end{equation*}
\end{theorem}

\begin{remark}
Note that $s\in(0,+\infty)$ is the maximal interval in which $A^{(s)}$ is of trace class. Indeed, $A^{(0)}=A$ is \emph{not} of trace class in general, as demonstrated in \cite[Sections~6 and~7]{curl}.
\end{remark}

\

Next, one observes that for $s=0$ the operator $A^{(s)}$ turns into the asymmetry operator $A$ --- cf.~\eqref{definition asymmetry operator equation} and~\eqref{definition parameter dependent asymmetry operator equation 1}. Furthermore, comparing
 \eqref{Aprin} with~\eqref{Theorem order of A(s) equation 1} we see that
\begin{equation*}
\label{Asprin vs Aprin}
(A^{(s)})_\mathrm{prin}
=
f(s)\,A_\prin
\qquad
\text{with}
\qquad
f(s)=
\frac{(s+1)(s+3)}{3\,\|\xi\|^s}
\,.
\end{equation*}

Proving Theorem~\ref{theorem trace A(s)}, which in turn will imply Theorem~\ref{Theorem main result}, requires one to carefully examine the behaviour of the integral kernel of $A^{(s)}$ as $s\to 0^+$. Let us emphasise that there are several nontrivial (and somewhat subtle) obstacles one needs to overcome in order to achieve this. In particular, one has to find a way to ``follow the singularity'' of $\mathfrak{a}^{(s)}$ up to $s=0^+$ in such a way that the error terms brought about by the microlocal approach do not grow in an uncontrolled fashion when the parameter $s$ becomes smaller and smaller.

\subsection*{Structure of the paper}
\addcontentsline{toc}{subsection}{Structure of the paper}

Our paper is structured as follows.

In Section~\ref{The matrix trace of an operator, revisited} we revisit the notion of matrix trace of an operator, building on the analysis from \cite[Section~4]{curl}. In particular, we provide stronger results for differential (as opposed to pseudodifferential) matrix operators.

In Section~\ref{The parameter-dependent asymmetry operator} we perform a detailed examination of the parameter-dependent asymmetry operator $A^{(s)}$: we discuss its basic properties, prove that it is a pseudodifferential operator of order $-s-3$, and compute its principal symbol.

Section~\ref{From As to eta} establishes a precise relationship between the operator $A^{(s)}$ and the local and global eta functions of the operator curl. This prepares the ground for Section~\ref{The leading singularity of the integral kernel of As}, which features a careful analysis of the leading singularity of the integral kernel of $A^{(s)}$. The latter is achieved by defining a reference operator which captures the leading singularity and underpins our notion of (regularised) trace for $A^{(s)}$.

Finally, in Section~\ref{Proof of Theorem main result} we give the proof of our main result, Theorem~\ref{Theorem main result}

Our paper is complemented by a section with concluding remarks, Section~\ref{Concluding remarks}, and three appendices with auxiliary technical material.

\section{The matrix trace of an operator, revisited}
\label{The matrix trace of an operator, revisited}

In this section we will briefly summarise and revisit the notion of \emph{matrix trace} of a pseudodifferential operator acting on 1-forms introduced in \cite[Section~4]{curl}. We refer the reader to \cite[Section~4]{curl} for further details and a broader discussion.

\

Consider a pseudodifferential operator
\begin{equation*}
\label{Q acting on 1-forms 1bis}
Q: u_\alpha(x)\mapsto\int_M \mathfrak{q}_\alpha{}^\beta(x,y)\,u_\beta(y)\,\rho(y)\,\dr y
\end{equation*}
of order $s$ acting on 1-forms, with scalar integral kernel $\mathfrak{q}_\alpha{}^\beta$
(Schwartz kernel).

There are two natural notions of trace associated with $Q$. The first notion is that of \emph{operator trace} $\operatorname{Tr}Q$, which is well defined when the operator is of trace class. This is the case if $s<-d$
and the operator is self-adjoint.
The second notion is that of \emph{matrix trace} $\operatorname{\mathfrak{tr}} Q$, introduced in Definition~\ref{definition matrix trace introduction}. The latter produces a \emph{scalar} pseudodifferential operator of the same order as the original operator~$Q$.

It was shown in \cite[Section~4]{curl} that the matrix trace satisfies the following properties.
\begin{enumerate}[(i)]
\item
The operator $\operatorname{\mathfrak{tr}} Q$ is defined uniquely, modulo the addition of a scalar integral operator whose integral kernel is infinitely smooth and vanishes in a neighbourhood of the diagonal.

\item
$(\operatorname{\mathfrak{tr}}Q)^*=\operatorname{\mathfrak{tr}}(Q^*)$, where the star refers to formal adjoints with respect to the natural inner products.
In particular, if $Q$ is self-adjoint, then so is $\operatorname{\mathfrak{tr}}Q$.

\item
If $s<-d$ and the operator $Q$ is self-adjoint then
\begin{equation}
\label{properties of matrix trace 1}
\operatorname{Tr}(\operatorname{\mathfrak{tr}} Q)=\operatorname{Tr}Q\,.
\end{equation}

\item
We have
\begin{equation}
\label{prin of tr is tr of prin}
(\operatorname{\mathfrak{tr}} Q)_\prin= \operatorname{tr} Q_\prin,
\end{equation}
\begin{equation}
\label{sub of tr is tr of sub}
(\operatorname{\mathfrak{tr}} Q)_\sub= \operatorname{tr} Q_\sub\,.
\end{equation}

\item
The quantities on the left-hand sides of \eqref{properties of matrix trace 1}--\eqref{sub of tr is tr of sub} are independent of the parameter $\epsilon$ appearing in formula \eqref{Q acting on 1-forms 6}.

\end{enumerate}

When $Q:u_\alpha\mapsto Q_\alpha{}^\beta\,u_\beta$ is a differential operator acting on 1-forms,
one has a stronger result. In order to formulate this result,
let us write down the operator $Q$ in local coordinates:
\begin{equation}
\label{differential operator Q in local coordinates}
Q_\alpha{}^\beta=
\sum_{|\boldsymbol{\kappa}|\le m}
[q^{\boldsymbol{\kappa}}]_\alpha{}^\beta(x)\,
\frac{\partial^{|\boldsymbol{\kappa}|}}{\partial x^{\boldsymbol{\kappa}}}
\,,
\end{equation}
where $m$ is the order of the differential operator,
$\boldsymbol{\kappa}=(\kappa_1, \dots,\kappa_d)\in\mathbb{N}_0^d$ is a multi-index,
$|\boldsymbol{\kappa}|=\sum_{j=1}^d \kappa_j$
and
$
\displaystyle
\frac{\partial^{|\boldsymbol{\kappa}|}}{\partial x^{\boldsymbol{\kappa}}}
=
\frac{\partial^{|\boldsymbol{\kappa}|}}
{\partial(x^1)^{\kappa_1}\dots\partial(x^d)^{\kappa_d}}
\,$.
Then introduce another differential operator
\begin{equation}
\label{differential operator R in local coordinates}
R_\alpha{}^\beta(x):=
\sum_{|\boldsymbol{\kappa}|\le m}
[q^{\boldsymbol{\kappa}}]_\alpha{}^\beta(x)\,
\frac{\partial^{|\boldsymbol{\kappa}|}}{\partial y^{\boldsymbol{\kappa}}}
\end{equation}
which acts in the variable $y$,
whereas the variable $x$ plays the role of parameter.

Recall that
\begin{equation}
\label{Q acting on 1-forms 3}
Z: T_xM\ni u^\alpha\mapsto u^\alpha\,Z_\alpha{}^\beta(x,y)\in T_yM
\end{equation}
is the linear map realising parallel transport of vectors from $x$ to $y$ along the unique shortest geodesic connecting $x$ and $y$. In what follows, we raise and lower indices in the 2-point tensor $Z_\alpha{}^\beta(x,y)$ using the Riemannian metric $g(x)$ in the first index and $g(y)$ in the second.

Recall also the identity \cite[formula~(4.9)]{curl}
\begin{equation}
\label{aux4}
Z_\beta{}^\alpha(y,x)
=
Z^\alpha{}_\beta(x,y).
\end{equation}

\begin{proposition}
\label{proposition trace of differential is differential}
Let $Q$ be a differential operator acting on 1-forms
and $R(x)$ be the corresponding parameter-dependent differential operator
defined, in local coordinates, in accordance with formula
\eqref{differential operator R in local coordinates}.
Then $\operatorname{\mathfrak{tr}}Q$ is the scalar differential operator
\begin{equation}
\label{23 November 2023 equation 1}
\operatorname{\mathfrak{tr}}Q:
f(x)
\mapsto
\left.
[
R_{\alpha}{}^\beta(x)\,
Z_\beta{}^\alpha(y,x)\,
f(y)
]
\right|_{y=x}\,.
\end{equation}
\end{proposition}

\begin{proof}
Let $U\subseteq M$ be a coordinate patch and let $x$ be local coordinates in $U$.
We assume that $U$ is small enough so that for all $x,y\in U$ we have
$\chi(\operatorname{dist}(x,y)/\epsilon)=1$,
where $\epsilon$ is the parameter appearing in formula \eqref{Q acting on 1-forms 6}.

In what follows, without loss of generality, the infinitely smooth 1-form $u$ and scalar function $f$ are assumed to be compactly supported in $U$.
This is acceptable because differential operators are local,
unlike the more general pseudodifferential operators.

In our coordinate patch and chosen local coordinates the differential operators $Q_{\alpha}{}^\beta$ read \eqref{differential operator Q in local coordinates}.
The operator $Q$ can now be equivalently rewritten in integral form as
\begin{equation}
\label{differential operator Q in local coordinates written as pseudodifferential operator}
Q:u_\alpha(x)
\mapsto
\frac{1}{(2\pi)^d}
\int
e^{i(x-y)^\gamma\xi_\gamma}\,
q_\alpha{}^\beta(x,\xi)\,
u_\beta(y)\,\dr y\,\dr\xi
\,,
\end{equation}
where
\begin{equation*}
\label{left symbol of the differential operator Q }
q_\alpha{}^\beta(x,\xi)
=
\sum_{|\boldsymbol{\kappa}|\le m}
i^{|\boldsymbol{\kappa}|}\,
[q^{\boldsymbol{\kappa}}]_\alpha{}^\beta(x)\,
\xi_{\boldsymbol{\kappa}}
\end{equation*}
is the (left) symbol of $Q$.
Here
$\xi_{\boldsymbol{\kappa}}=(\xi_1)^{\kappa_1}\dots(\xi_d)^{\kappa_d}$.

Formulae
\eqref{Q acting on 1-forms 1}--\eqref{Q acting on 1-forms 6}
and
\eqref{differential operator Q in local coordinates written as pseudodifferential operator}
imply
\begin{equation}
\label{6 October 2024 Dima equation 1}
\operatorname{\mathfrak{tr}}Q:
f(x)
\mapsto
\frac{1}{(2\pi)^d}
\int
e^{i(x-y)^\gamma\xi_\gamma}\,
q_\alpha{}^\beta(x,\xi)\,
Z_\beta{}^\alpha(y,x)\,
f(y)\,\,\dr y\,\dr\xi
\,.
\end{equation}
%But \cite[formula~(4.9)]{curl} allows us to rewrite
%\eqref{6 October 2024 Dima equation 1} as
%\begin{equation}
%\label{6 October 2024 Dima equation 2}
%\operatorname{\mathfrak{tr}}Q:
%f(x)
%\mapsto
%\frac{1}{(2\pi)^d}
%\int
%e^{i(x-y)^\gamma\xi_\gamma}\,
%q_\alpha{}^\beta(x,\xi)\,
%Z^\alpha{}_\beta(x,y)\,
%f(y)\,\,\dr y\,\dr\xi
%\,.
%\end{equation}
Examination of formula \eqref{6 October 2024 Dima equation 1}
shows that we are looking at \eqref{23 November 2023 equation 1}.
\end{proof}

\

We feel it necessary to provide an alternative equivalent reformulation of
Proposition~\ref{proposition trace of differential is differential},
one that avoids the use of the parameter-dependent differential operator
\eqref{differential operator R in local coordinates}.
Let us instead make use of 
the amplitude-to-symbol operator
\begin{equation}
\label{proof subprincipal symbol operators 1 forms equation 6}
\mathcal{S}_\mathrm{left}\sim \sum_{k=0}^{+\infty} \mathcal{S}_\mathrm{-k}, \qquad \mathcal{S}_0:=\left.(\,\cdot\,)\right|_{y=x}, \qquad \mathcal{S}_{-k}:=\frac{1}{k!}\left. \left[ \left(-i\dfrac{\partial^2}{\partial y^\gamma\partial\xi_\gamma} \right)^k(\,\cdot\,)\right] \right|_{y=x},
\end{equation}
see \cite[formula~(3.13)]{curl}.
The action of the operator
\eqref{proof subprincipal symbol operators 1 forms equation 6}
on an amplitude $q(x,y,\xi)$ of a (pseudo)\-differential operator
excludes the $y$-dependence, i.e.~gives the (left) symbol.
Note that the operator
\eqref{proof subprincipal symbol operators 1 forms equation 6}
maps an amplitude polynomial in $\xi$ to a (left) symbol polynomial in $\xi$.

\begin{proposition}
\label{proposition trace of differential is differential alternative}
Let $Q$ be a differential operator acting on 1-forms
with (left) symbol $q_\alpha{}^\beta(x,\xi)$.
Then $\operatorname{\mathfrak{tr}}Q$ is the scalar differential operator
with (left) symbol
\begin{equation}
\label{23 November 2023 equation 1 alternative}
\mathcal{S}_\mathrm{left}\,
[
q_\alpha{}^\beta(x,\xi)\,
Z_\beta{}^\alpha(y,x)
]\,.
\end{equation}
\end{proposition}

\begin{proof}
Consequence of formula \eqref{6 October 2024 Dima equation 1}
and the definition of the amplitude-to-symbol operator.
\end{proof}

\begin{remark}
It is easy to see that if
$Q$ is a pseudodifferential operator acting on 1-forms
with (left) symbol $q_\alpha{}^\beta(x,\xi)$,
then $\operatorname{\mathfrak{tr}}Q$ is the scalar pseudodifferential operator
with (left) symbol given by formula
\eqref{23 November 2023 equation 1 alternative}.
Of course, in the pseudodifferential case symbols are not polynomials in $\xi$
and are defined modulo $O(|\xi|^{-\infty})$ as $|\xi|\to+\infty$.
\end{remark}

\

To conclude this section, we will examine some important examples of matrix trace.
In preparation for this analysis, recall
%\cite[formula~(6.33)]{curl}
that in geodesic normal coordinates centred at $x=0$ the metric admits the expansion
\begin{equation}
\label{7 November 2023 equation 15}
g_{\alpha\beta}(y)
=
\delta_{\alpha\beta}
-
\frac13
\Riem_{\alpha\mu\beta\nu}(0)\,y^\mu y^\nu
%-
%\frac16
%(\nabla_\sigma \Riem_{\alpha\mu\beta\nu})(0)\,y^\sigma\,y^\mu\,y^\nu
%+O(|y|^4)\,,
+O(|y|^3)\,,
\end{equation}
so that one has
$\,
\displaystyle
\det g_{\alpha\beta}(y)
=
1
-
\frac13
\Ric_{\mu\nu}(0)\,y^\mu y^\nu
%-
%\frac16
%(\nabla_\sigma \Ric_{\mu\nu})(0)\,y^\sigma\,y^\mu\,y^\nu
%+O(|y|^4)
+O(|y|^3)
\,$
and
\begin{equation}
\label{7 November 2023 equation 17}
\rho(y)
=
1
-
\frac16
\Ric_{\mu\nu}(0)\,y^\mu y^\nu
%-
%\frac1{12}
%(\nabla_\sigma \Ric_{\mu\nu})(0)\,y^\sigma\,y^\mu\,y^\nu
%+O(|y|^4)\,.
+O(|y|^3)\,.
\end{equation}
Furthermore, the parallel transport map $Z^\alpha{}_\beta(x,y)$ admits the following expansion \cite[formula~(E.2)]{curl}
\begin{equation}
\label{prop: expansion of parallel transoport equation 2}
Z^\alpha{}_\beta(0,y)=\delta^\alpha{}_\beta-\frac16 \Riem^\alpha{}_{\mu\beta\nu}(0)\,y^\mu y^\nu +\frac16 \frac{\partial^2\Gamma^\alpha{}_{\mu\beta}}{\partial y^\nu \partial y^\rho}(0) \,y^\mu y^\nu y^\rho+O(|y|^4)\,.
\end{equation}

\begin{proposition}
\label{proposition trace curl}
In dimension $d=3$ we have
\begin{equation}
\label{proposition trace curl equation 1}
\operatorname{\mathfrak{tr}} \curl=0\,.
\end{equation}
\end{proposition}

\begin{proof}
Formula \eqref{23 November 2023 equation 1} applied to $Q=\curl$ gives us
\begin{equation}
\label{proof proposition trace curl equation 1}
\operatorname{\mathfrak{tr}} \curl :
f(x)
\mapsto
\left.
\left[
-E_\alpha{}^{\beta\gamma}(x)\,\frac{\partial}{\partial y^\gamma}\,
Z_\beta{}^\alpha(y,x)\,
f(y)
\right]
\right|_{y=x}\,.
\end{equation}
Let us choose geodesic normal coordinates centred at $x=0$. Then formulae
\eqref{proof proposition trace curl equation 1},
\eqref{Levi-Civita tensor},
\eqref{7 November 2023 equation 15},
\eqref{7 November 2023 equation 17},
\eqref{aux4}
and 
\eqref{prop: expansion of parallel transoport equation 2}
immediately imply \eqref{proposition trace curl equation 1}.
\end{proof}

\begin{proposition}
\label{proposition trace curl cubed}
In dimension $d=3$ we have
\begin{equation*}
\label{proposition trace curl cubed equation 1}
\operatorname{\mathfrak{tr}} (\curl^3)=0\,.
\end{equation*}
\end{proposition}

\begin{proof}
We are looking at the differential operator $Q=\curl^3\,$,
\begin{equation}
\label{proof proposition trace curl cubed equation 2}
Q_\alpha{}^\beta
=-E_\alpha{}^{\gamma\rho}(x)\,\frac{\partial}{\partial x^\rho}
\left(
E_\gamma{}^{\mu\sigma}(x)\,\frac{\partial}{\partial x^\sigma}
\left(
E_\mu{}^{\beta\tau}(x)\,\frac{\partial}{\partial x^\tau}
\right)
\right).
\end{equation}
We need to rewrite
\eqref{proof proposition trace curl cubed equation 2}
in the form
\eqref{differential operator Q in local coordinates},
i.e.~put all the coefficients in front of partial derivatives,
and then form the parameter-dependent differential operator $R(x)$
in accordance with formula
\eqref{differential operator R in local coordinates}.

Let us choose geodesic normal coordinates centred at $x=0$. A somewhat lengthy but straightforward calculation shows that in the chosen coordinate system the differential operator $R(0)$ reads
\begin{multline}
\label{proof proposition trace curl cubed equation 3 corrected 2}
R_\alpha{}^\beta(0)
=
\left[
\varepsilon^{\beta\rho\tau}
\Ric_{\alpha\rho}(0)
-
\frac13\varepsilon_\alpha{}^{\beta\rho}
\Ric_{\rho}{}^\tau(0)\,
+
\frac13\varepsilon_\alpha{}^{\tau\rho}
\Ric_{\rho}{}^\beta(0)
\right]
\frac{\partial}{\partial y^\tau}
\\
+
\varepsilon_\alpha{}^{\beta\rho}\,
\frac{\partial}{\partial y^\rho}
\left(
\delta^{\sigma\tau}
\frac{\partial^2}{\partial y^\sigma\partial y^{\tau}}
\right).
\end{multline}
In writing down \eqref{proof proposition trace curl cubed equation 3 corrected 2} we used
\eqref{Levi-Civita tensor},
\eqref{7 November 2023 equation 15},
\eqref{7 November 2023 equation 17} and the elementary identity
\begin{equation*}
\label{proof proposition trace curl cubed equation 4}
\varepsilon_{\alpha\beta\gamma}\,\varepsilon^{\mu\nu\gamma}= \delta_\alpha{}^\mu \delta_\beta{}^\nu-\delta_\alpha{}^\nu \delta_\beta{}^\mu.
\end{equation*}

It only remains to substitute
\eqref{proof proposition trace curl cubed equation 3 corrected 2},
\eqref{aux4}
and 
\eqref{prop: expansion of parallel transoport equation 2}
into the RHS of
\eqref{23 November 2023 equation 1}
at $x=0$.

We consider separately the contributions to 
\eqref{23 November 2023 equation 1}
from the three terms on the RHS of
\eqref{prop: expansion of parallel transoport equation 2}.

The contribution from the first term on the RHS of
\eqref{prop: expansion of parallel transoport equation 2}
vanishes because $\,R_\alpha{}^\beta(0)\,\delta^\alpha{}_\beta=0\,$.

The contribution from the second term on the RHS of
\eqref{prop: expansion of parallel transoport equation 2} reads
\begin{equation*}
\label{contribution 2 formula 1}
-\,\frac{1}{3}\,
\varepsilon_\alpha{}^{\beta\rho}\,
\Ric^\alpha{}_\beta(0)\,
\left.\frac{\partial f}{\partial y^\rho}\right|_{y=0}
-\,\frac{1}{3}\,
\varepsilon_\alpha{}^{\beta\rho}\,
[\Riem^\alpha{}_{\rho\beta}{}^\sigma(0)+\Riem^{\alpha\sigma}{}_{\beta\rho}(0)]
\left.\frac{\partial f}{\partial y^\sigma}\right|_{y=0}\,.
\end{equation*}
The tensors
$\,\Ric^{\alpha\beta}(0)\,$
and
$\,\Riem^{\alpha\rho\beta\sigma}(0)+\Riem^{\alpha\sigma\beta\rho}(0)\,$
are symmetric in $\alpha,\beta$,
hence the expression \eqref{contribution 2 formula 1} vanishes.

Note that the fact that
the second term on the RHS of
\eqref{prop: expansion of parallel transoport equation 2} 
gives a zero contribution to the RHS of
\eqref{23 November 2023 equation 1}
at $x=0$ can be established without involving the
explicit formula \eqref{contribution 2 formula 1}.
In dimension three the Riemann curvature tensor is expressed via the Ricci tensor,
so we are looking at a linear combination of terms of the form
\begin{equation*}
\label{contribution 2 formula 2}
\varepsilon_{\alpha\beta\gamma}
\times
\Ric_{\kappa\lambda}
\times
\left.\frac{\partial f}{\partial y^\mu}\right|_{y=0}\,.
\end{equation*}
Here one has to perform contraction of indices to get a scalar.
It is easy to see that any such contraction gives zero.

The proof of Proposition~\ref{proposition trace curl cubed} has been reduced to
examining the contribution to 
\eqref{23 November 2023 equation 1}
coming from the third term on the RHS of
\eqref{prop: expansion of parallel transoport equation 2}.
Namely, the task at hand is to show that
\begin{equation*}
\label{proof proposition trace curl cubed equation corrected}
\varepsilon_\alpha{}^{\beta\rho}\,
\frac{\partial}{\partial y^\rho}
\left(
\delta^{\sigma\tau}
\frac{\partial^2}{\partial y^\sigma\partial y^{\tau}}
\right)
\frac{\partial^2\Gamma^\alpha{}_{\mu\beta}}{\partial y^\nu \partial y^\gamma}(0) \,y^\mu y^\nu y^\gamma
=0\,.
\end{equation*}
But $\partial^2\Gamma$ is expressed via $\nabla\Ric\,$,
so we are looking at a linear combination of terms of the form
\begin{equation*}
\label{contribution 3 formula 2}
\varepsilon_{\alpha\beta\gamma}
\times
\nabla_\kappa\Ric_{\lambda\mu}
\,.
\end{equation*}
Here one has to perform contraction of indices to get a scalar.
It is easy to see that any such contraction gives zero.
This completes the proof of Proposition~\ref{proposition trace curl cubed}.
\end{proof}

\begin{proposition}
\label{proposition trace Hodge Laplacian}
In any dimension $d\ge2$ we have
\begin{equation}
\label{proposition trace Hodge Laplacian equation 1}
\operatorname{\mathfrak{tr}} \boldsymbol{\Delta}=d\Delta-\Sc\,,
\end{equation}
where $\boldsymbol{\Delta}$ is the (nonpositive) Hodge Laplacian acting on 1-forms,
$\Delta$ is the (nonpositive) Laplace--Beltrami operator
and $\Sc$ is scalar curvature.
\end{proposition}

\begin{proof}
The (nonpositive) Hodge Laplacian $\boldsymbol{\Delta}$ acting on 1-forms is defined by formula
\begin{equation}
\label{definition of Hodge Laplacian}
\boldsymbol{\Delta}=-(\dr\delta+\delta\dr)\,.
\end{equation}
In local coordinates we have
\begin{equation}
\label{8 December 2023 equation 5}
(\dr\delta)_{\alpha}{}^\beta=
-
\frac{\partial}{\partial x^\alpha}\,
\rho^{-1}\,
\frac{\partial}{\partial x^\gamma}\,
g^{\gamma\beta}\,\rho\,,
\end{equation}
\begin{equation}
\label{8 December 2023 equation 3}
(\delta \dr)_{\alpha}{}^\beta=
g_{\alpha\nu}\,
\rho^{-1}\,
\frac{\partial}{\partial x^\mu}
\left(
g^{\mu\beta}\,g^{\nu\gamma}
-
g^{\mu\gamma}\,g^{\nu\beta}
\right)
\rho\,
\frac{\partial}{\partial x^\gamma}\,.
\end{equation}
In \eqref{8 December 2023 equation 5} and \eqref{8 December 2023 equation 3} it is understood that partial derivatives act on everything to their right.

We are looking at the differential operator $Q=\boldsymbol{\Delta}$
defined by formulae
\eqref{definition of Hodge Laplacian}--\eqref{8 December 2023 equation 3}.
We need to rewrite it
in the form
\eqref{differential operator Q in local coordinates},
i.e.~put all the coefficients in front of partial derivatives,
and then form the parameter-dependent differential operator $R(x)$
in accordance with formula
\eqref{differential operator R in local coordinates}.

As in the proof of Proposition~\ref{proposition trace curl cubed},
let us choose geodesic normal coordinates centred at $x=0$.
A~straightforward calculation based on the use of formulae
\eqref{7 November 2023 equation 15}
and
\eqref{7 November 2023 equation 17}
shows that in the chosen coordinate system the differential operator $R(0)$ reads
\begin{equation}
\label{R(0) for Hodge Laplacian}
R_\alpha{}^\beta(0)
=
\delta_\alpha{}^\beta
\left(
\delta^{\mu\gamma}
\frac{\partial^2}{\partial y^\mu\partial y^{\gamma}}
\right)
-
\frac23
\Ric_\alpha{}^\beta(0)\,.
\end{equation}

We now substitute
\eqref{R(0) for Hodge Laplacian},
\eqref{aux4}
and 
\eqref{prop: expansion of parallel transoport equation 2}
into the RHS of
\eqref{23 November 2023 equation 1}
at $x=0$.

We consider separately the contributions to 
\eqref{23 November 2023 equation 1}
from the first two terms on the RHS of
\eqref{prop: expansion of parallel transoport equation 2}.

The contribution from the first term on the RHS of
\eqref{prop: expansion of parallel transoport equation 2}
reads
$\displaystyle
\,
d
\left(
\delta^{\mu\gamma}
\frac{\partial^2}{\partial y^\mu\partial y^{\gamma}}
\right)
-
\frac23
\Sc(0)\,
$,
whereas the contribution from the second term on the RHS of
\eqref{prop: expansion of parallel transoport equation 2}
reads
$\displaystyle
\,
-
\frac13
\Sc(0)\,
$.
Hence,
in geodesic normal coordinates centred at $x=0$
we have
\begin{equation}
\label{Dima 28 October 2024 equation 1}
[\,(\operatorname{\mathfrak{tr}}\boldsymbol{\Delta})f\,](0)
=
\left.
\left[
\,
d
\left(
\delta^{\mu\gamma}
\frac{\partial^2 f}{\partial y^\mu\partial y^{\gamma}}
\right)
-
\Sc f\,
\right]
\right|_{y=0}\,.
\end{equation}

Comparing \eqref{Dima 28 October 2024 equation 1}
with the action of the Laplace--Beltrami operator
$
\displaystyle
\,
\Delta
=
\rho^{-1}\,
\frac{\partial}{\partial y^\mu}\,
g^{\mu\nu}\,\rho\,
\frac{\partial}{\partial y^\nu}
\,$,
we arrive at \eqref{proposition trace Hodge Laplacian equation 1}.
\end{proof}

\section{The parameter-dependent asymmetry operator $A^{(s)}$}
\label{The parameter-dependent asymmetry operator}

In this section we will examine the parameter-dependent asymmetry operator $A^{(s)}$ introduced in Definition~\ref{definition parameter dependent asymmetry operator} and establish its basic properties. In particular, we will prove Theorem~\ref{Theorem order of A(s)}.

\subsection{An alternative representation for $A^{(s)}$}
\label{An alternative representation for As}

To start with, let us provide an alternative representation for the operator $A^{(s)}$.

Let
\begin{equation*}
\label{eigenvalues of Laplace--Beltrami}
0=\mu_0<\mu_1\le \mu_2\le \dots \to +\infty
\end{equation*}
be the eigenvalues of the operator $-\Delta$ enumerated in increasing order and with account of multi\-plicity, and let $f_j$, $j=0,1,2,\dots$, be the corresponding orthonormalised eigenfunctions. Here $\Delta$ is the (nonpositive) Laplace--Beltrami operator.

Recall also that $(\lambda_j,u_j)$, $j\in \mathbb{Z}\setminus \{0\}$, is our notation the eigensystem of $\curl$.

\begin{lemma}
\label{lemma basic properties As}
\phantom{a}
\begin{enumerate}[(a)]
\item
We have
\begin{equation}
\label{alternative representation for As}
A^{(s)}=\operatorname{\mathfrak{tr}}\left[\curl\, (-\boldsymbol{\Delta})^{-\frac{s+1}2}\right].
\end{equation}

\item
If $s>3$ the operator $A^{(s)}$ is of trace class.
\end{enumerate}
\end{lemma}

\begin{proof}
(a)
The family of 1-forms $v_j:=\mu_j^{-1/2}\dr f_j$, $j=1,2,\dots$, forms an orthonormal basis for the Hilbert space $\dr\Omega^0$ with inner product \eqref{inner product}.
Since on a Riemannian 3-manifold the codifferential $\delta$ acts on $\Omega^k$ as
$\,\delta=(-1)^k *\dr\,*\,$,
%\begin{equation}
%\label{codifferential on k-forms via hodge}
%\delta=(-1)^k *\dr\,*\,,
%\end{equation}
the Spectral Theorem gives us the following representation for the (nonpositive) Hodge Laplacian
\eqref{definition of Hodge Laplacian}
acting on 1-forms:
\begin{equation}
\label{spectral decomposition Hodge Laplacian}
\boldsymbol{\Delta}=-(\dr\delta+\delta\dr)=-\sum_{j=1}^{+\infty}\mu_j\,v_j\langle v_j,\,\cdot\,\rangle-\sum_{j\in\mathbb{Z}\setminus\{0\}}\lambda_j^2\,u_j\langle u_j,\,\cdot\,\rangle\,.
\end{equation}
Hence,
\begin{equation}
\label{tth power of Hodge Laplacian}
(-\boldsymbol{\Delta})^r=
(\dr\delta+\delta\dr)^r=\sum_{j=1}^{+\infty}\mu_j^r\,v_j\langle v_j,\,\cdot\,\rangle+\sum_{j\in\mathbb{Z}\setminus\{0\}}|\lambda_j|^{2r}\,u_j\langle u_j,\,\cdot\,\rangle\,,
\qquad r\in\mathbb{R}\,.
\end{equation}

Formulae
\eqref{definition Pplus},
\eqref{definition Pminus} and
\eqref{tth power of Hodge Laplacian} and Hodge's decomposition imply
\begin{equation}
\label{Pplus-Pminus alterntive representation}
P_+-P_-=\curl\,(-\boldsymbol{\Delta})^{-1/2}\,.
\end{equation}
Substituting \eqref{Pplus-Pminus alterntive representation} into~\eqref{definition parameter dependent asymmetry operator equation 1} we obtain \eqref{alternative representation for As}.

(b)
Examination
of formula \eqref{definition parameter dependent asymmetry operator equation 1}
or formula \eqref{alternative representation for As} shows that for $s>3$
the order of the pseudo\-differential operator $A^{(s)}$ is strictly less than $\,-3\,$.
But a self-adjoint pseudodifferential operator of order strictly less than $\,-3\,$ is of trace class (recall that the dimension of $M$ is 3).
\end{proof}

Although at first glance the meaning of~\eqref{alternative representation for As} is less transparent than that of~\eqref{definition parameter dependent asymmetry operator equation 1}, the representation \eqref{alternative representation for As} will be especially convenient in the forthcoming calculations because it involves the composition of a power of the Hodge Laplacian with the differential operator $\curl$, as opposed to the pseudodifferential operator $P_+-P_-$. Indeed, from a microlocal perspective $\curl$ possesses the following nice properties upon which we will rely extensively:
\begin{enumerate}[(i)]
\item the full symbol of $\curl$ coincides with its principal symbol \eqref{principal symbol curl} and

\item
the symbol of $\curl$ is linear in momentum $\xi\,$, so that
\begin{equation}
\label{higher derivatives of curl prin are zero}
\frac
{\partial^{|\boldsymbol{\kappa}|}\,[\curl_\prin]_{\alpha}{}^\beta}
{\partial\xi_{\boldsymbol{\kappa}}}
(x,\xi)=0
\end{equation}
for all multi-indices $\boldsymbol{\kappa}\in \mathbb{N}_0^3$ with $|\boldsymbol{\kappa}|\ge 2$.
\end{enumerate}

\subsection{The order of $A^{(s)}$}
\label{The order of As}

The goal of this subsection is to prove part (a) of Theorem~\ref{Theorem order of A(s)}. Namely, we will show that the operator $A^{(s)}$ is of order $-s-3$. This will be achieved in several steps, decreasing the order of $A^{(s)}$ at each step.

\

The first step is the simplest, in that it does not require extensive use of the resolvent.

\begin{lemma}
\label{lemma As is of order -s-2}
The parameter-dependent asymmetry operator $A^{(s)}$ is a pseudodifferential operator of order $-s-2$.
\end{lemma}

\begin{proof}
The claim is equivalent to the following two statements:
\begin{equation}
\label{proof lemma As is of order -s-2 equation 1}
(A^{(s)})_\prin=0\,,
\end{equation}
\begin{equation}
\label{proof lemma As is of order -s-2 equation 2}
(A^{(s)})_\sub=0\,.
\end{equation}
Here the subprincipal symbol of a pseudodifferential operator acting on 1-forms is defined in accordance with \cite[Definition~3.2]{curl}.

Formula \eqref{proof lemma As is of order -s-2 equation 1} follows from \eqref{alternative representation for As}, \eqref{prin of tr is tr of prin}, \eqref{principal symbol curl}
%the fact that
%\begin{equation}
%\label{proof lemma As is of order -s-2 equation 3}
%[(P_\pm)_\mathrm{prin}]_\alpha{}^\beta(x,\xi)
%=
%\frac12
%\left[
%\delta_\alpha{}^\beta
%-
%\|\xi\|^{-2}\,\xi_\alpha \,g^{\beta\gamma}(x)\,\xi_\gamma
%\pm i\,
%\|\xi\|^{-1}\,
%E_\alpha{}^{\gamma\beta}(x)\,
%\xi_\gamma
%\right],
%\end{equation}
and the fact that
\begin{equation}
\label{proof lemma As is of order -s-2 equation 4}
[(-\boldsymbol{\Delta})^{-\frac{s+1}2}]_\prin=\|\xi\|^{-\frac{s+1}2}\, \mathrm{I}
\end{equation}
is proportional to the identity matrix $\operatorname{I}\,$.

Theorem~3.8 from \cite{curl} tells us that
\begin{multline}
\label{proof lemma As is of order -s-2 equation 5}
[\curl\, (-\boldsymbol{\Delta})^{-\frac{s+1}2}]_\sub=\curl_\prin [(-\boldsymbol{\Delta})^{-\frac{s+1}2}]_\sub+\curl_\sub [(-\boldsymbol{\Delta})^{-\frac{s+1}2}]_\prin
\\
+
\frac{i}2\{\{\curl_\mathrm{prin},[(-\boldsymbol{\Delta})^{-\frac{s+1}2}]_\prin\}\}\,,
\end{multline}
where 
\begin{multline}
\label{proof lemma As is of order -s-2 equation 6}
\{\{Q_\mathrm{prin},R_\mathrm{prin}\}\}_\alpha{}^\beta
:=
\left(
\frac{\partial [Q_\mathrm{prin}]_\alpha{}^\kappa}{\partial x^\gamma}
-
\Gamma^{\alpha'}{}_{\gamma\alpha}[Q_\mathrm{prin}]_{\alpha'}{}^\kappa
+
\Gamma^\kappa{}_{\gamma\kappa'}[Q_\mathrm{prin}]_\alpha{}^{\kappa'}
\right)
\frac{\partial[R_\mathrm{prin}]_\kappa{}^\beta}{\partial\xi_\gamma}
\\
-
\frac{\partial[Q_\mathrm{prin}]_\alpha{}^\kappa}{\partial\xi_\gamma}
\left(
\frac{\partial [R_\mathrm{prin}]_\kappa{}^\beta}{\partial x^\gamma}
-
\Gamma^{\kappa'}{}_{\gamma\kappa}[R_\mathrm{prin}]_{\kappa'}{}^\beta
+
\Gamma^\beta{}_{\gamma\beta'}[R_\mathrm{prin}]_\kappa{}^{\beta'}
\right)
\end{multline}
is the generalised Poisson bracket. Furthermore, \cite[Lemma~3.6]{curl} tells us that 
\begin{equation}
\label{proof lemma As is of order -s-2 equation 7}
\curl_\sub=0
\end{equation}
and Lemma~\ref{lemma subprincipal symbol powers of the Hodge Laplacian} tells us that 
\begin{equation}
\label{proof lemma As is of order -s-2 equation 8}
[(-\boldsymbol{\Delta})^{-\frac{s+1}2}]_\sub=0.
\end{equation}
Finally, a straightforward calculation involving \eqref{principal symbol curl}, \eqref{proof lemma As is of order -s-2 equation 4} and \eqref{proof lemma As is of order -s-2 equation 6} gives us
\begin{equation}
\label{proof lemma As is of order -s-2 equation 9}
\{\{\curl_\mathrm{prin},[(-\boldsymbol{\Delta})^{-\frac{s+1}2}]_\prin\}\}=0.
\end{equation}
Substituting \eqref{proof lemma As is of order -s-2 equation 7}--\eqref{proof lemma As is of order -s-2 equation 9} into \eqref{proof lemma As is of order -s-2 equation 5} we obtain
\begin{equation}
\label{proof lemma As is of order -s-2 equation 10}
[\curl\, (-\boldsymbol{\Delta})^{-\frac{s+1}2}]_\sub=0\,.
\end{equation}
Formula \eqref{proof lemma As is of order -s-2 equation 2} then follows from \eqref{alternative representation for As}, \eqref{sub of tr is tr of sub} and \eqref{proof lemma As is of order -s-2 equation 10}.
\end{proof}

In view of Lemma~\ref{lemma As is of order -s-2}, formula~\eqref{alternative representation for As}, and formula \eqref{prin of tr is tr of prin}, in order to establish that $A^{(s)}$ is of order $-s-3$ it suffices to compute the homogeneous component of the symbol of the operator $\curl\, (-\boldsymbol{\Delta})^{-\frac{s+1}{2}}$ of degree $-s-2$ and show that its pointwise matrix trace vanishes.

\

Let
\begin{equation}
\label{symbol q of Delta -s-1/2}
q_{\alpha}{}^\beta(x,\xi)\sim \|\xi\|^{-s-1}\,\delta_{\alpha}{}^\beta +[q_{-s-2}]_{\alpha}{}^\beta(x,\xi)+ [q_{-s-3}]_{\alpha}{}^\beta(x,\xi)+\dots
\end{equation}
be the full symbol of the operator $(-\boldsymbol{\Delta})^{-\frac{s+1}{2}}$. Then the full symbol $t_{\alpha}{}^\beta$ of the operator $\curl\, (-\boldsymbol{\Delta})^{-\frac{s+1}{2}}$ is given by
\begin{equation}
\label{symbol curl Delta -s-1/2}
t_{\alpha}{}^\beta(x,\xi) = E_\alpha{}^{\mu\gamma}(x) \left(i \xi_\mu\, q_{\gamma}{}^\beta(x,\xi) + \frac{\partial q_{\gamma}{}^\beta}{\partial x^\mu}(x,\xi) \right)\,.
\end{equation}
In writing \eqref{symbol curl Delta -s-1/2} we used \eqref{principal symbol curl}, \eqref{higher derivatives of curl prin are zero}, and the fact
that, given two pseudodifferential operators $A$ and $B$ with symbols $a$ and $b$, the symbol $\sigma_{AB}$ of their composition $AB$ is given by the formula
\begin{equation}
\label{5 November 2021 equation 1}
\sigma_{AB}\sim\sum_{k=0}^{\infty} \frac{1}{i^k k!}  \dfrac{\partial^k a}{\partial \xi_{\alpha_1}\dots \partial \xi_{\alpha_k}}\dfrac{\partial^k b}{\partial x^{\alpha_1}\dots \partial x^{\alpha_k}}\,,
\end{equation}
see \cite[Theorem~3.4]{shubin}. The identity \eqref{symbol curl Delta -s-1/2} effectively reduces the task at hand to the analysis of the symbol of $(-\boldsymbol{\Delta})^{-\frac{s+1}{2}}$, the power of an elliptic differential operator.

%\color{red}
%
%Eventually formula \eqref{prop: expansion of parallel transoport equation 3}
%will have to be removed because it already appeared in
%Section~\ref{The matrix trace of an operator, revisited}
%in the form of equations
%\eqref{prop: expansion of parallel transoport equation 2}
%and
%\eqref{aux4}.
%Will keep it around for a while.
%
%Recall that the parallel transport map admits the following expansion in normal coordinates \cite[Proposition~E.1]{curl}:
%\begin{equation}
%\label{prop: expansion of parallel transoport equation 3}
%Z_\alpha{}^\beta(y,0)=\delta_\alpha{}^\beta-\frac16 \Riem^\beta{}_{\mu\alpha\nu}(0)\,y^\mu y^\nu +\frac16 \frac{\partial^2\Gamma^\beta{}_{\mu\alpha}}{\partial y^\nu \partial y^\rho}(0) \,y^\mu y^\nu y^\rho+O(|y|^4)\,\,.
%\end{equation}
%
%\color{black}

\

Let us fix an arbitrary point $z\in M$ and work in geodesic normal coordinates centred at $z$. In our chosen coordinate system the operator $A^{(s)}$ (modulo $\Psi^{-\infty}$) reads
\begin{equation}
\label{As in normal coordinates}
A^{(s)}: f(x) \mapsto
\frac{1}{(2\pi)^3}
\int e^{i(x-y)^\mu\xi_\mu}\,
t_{\alpha}{}^\beta(x,\xi)
\,
Z_\beta{}^\alpha(y,x)
\,
f(y)
\,\dr y\,\dr\xi\,,
\end{equation}
where $\,t_{\alpha}{}^\beta(x,\xi)\,$ is defined in accordance with
\eqref{symbol curl Delta -s-1/2}
and
\eqref{symbol q of Delta -s-1/2}.

Next, we observe that \eqref{As in normal coordinates} can be recast as
\begin{equation}
\label{As in normal coordinates v2}
A^{(s)}=A^{(s)}_\mathrm{diag}+A^{(s)}_\mathrm{pt}\,,
\end{equation}
where 
\begin{equation}
\label{As diag}
A^{(s)}_\mathrm{diag}:=\frac{1}{(2\pi)^3}
\int e^{i(x-y)^\mu\xi_\mu}\,
t_{\alpha}{}^\alpha(x,\xi)
\,
f(y)
\,\dr y\,\dr\xi
\end{equation}
and 
\begin{equation}
\label{As pt}
A^{(s)}_\mathrm{pt}:=A^{(s)}-A^{(s)}_\mathrm{diag}.
\end{equation}
Here the subscript ``pt'' stands for ``parallel transport''. Let us emphasise that the decomposition \eqref{As in normal coordinates v2} is not invariant, i.e., it relies on our particular choice of local coordinates. However, it turns out to be very convenient when carrying out the explicit calculations below.

The following lemma establishes that
the contribution from $A^{(s)}_\mathrm{pt}$
to the principal symbol of the parameter-dependent asymmetry operator $A^{(s)}$ at the point $z$
can be disregarded.

\begin{lemma}
\label{lemma Apt}
Let $a^{(s)}_{\mathrm{pt}}(x,\xi)\sim \sum_{j=0}^{+\infty}(t_{\mathrm{pt}})_{-s-j}(x,\xi)$ be the full symbol of the operator 
$A^{(s)}_\mathrm{pt}\,$. We have
\begin{equation}
\label{lemma Apt equation 1}
(a^{(s)}_{\mathrm{pt}})_{-s-j}(z,\xi)=0 \quad \text{for}\quad j=0,1,2,3\,.
\end{equation}
\end{lemma}
\begin{proof}
The claim \eqref{lemma Apt equation 1} for $j=0$ and $j=1$ follows at once from
\eqref{prop: expansion of parallel transoport equation 2}
and
\eqref{aux4}.
The claim \eqref{lemma Apt equation 1} for $j=2$ and $j=3$ can be proved following the strategy from \cite[subsection~6.2]{curl}. Detailed arguments are given in Appendix~\ref{Proof of Lemma Apt}.
\end{proof}

\

Let us further prepare the ground for the final step in the proof of Theorem~\ref{Theorem order of A(s)}, part (a).

In what follows we assume, for simplicity, that
\begin{equation}
\label{s is greater than -1}
s>-1\,.
\end{equation}
The general case $s\in\mathbb{R}$ can be handled by means of \cite[Proposition~10.1]{shubin}.

Let
$
R_\lambda:=\left(-\boldsymbol{\Delta}-\,\lambda\,\mathrm{Id}\right)^{-1}
$
be the resolvent of $-\boldsymbol{\Delta}$. For fixed $(x,\xi)\in T^*M\setminus \{0\}$ and $\rho=\rho(x,\xi)<\|\xi\|^2$, let $\Gamma=\Gamma(x,\xi)=\Gamma_1\cup \Gamma_2 \cup \Gamma_3$ be the contour in the complex plane defined by
$$\Gamma_1:=\{\lambda \in \mathbb{C}\ | \ \lambda=re^{i\pi}, \, \rho<r<+\infty\},$$
$$\Gamma_2:=\{\lambda \in \mathbb{C}\ | \ \lambda=\rho \,e^{i\theta}, \, -\pi<\theta<\pi\},$$
$$\Gamma_3:=\{\lambda \in \mathbb{C}\ | \ \lambda=re^{-i\pi}, \, \rho<r<+\infty\},$$
with orientation prescribed as in the Figure~\ref{contour}.

\begin{figure}[h!]
\centering
\includegraphics[scale=1]{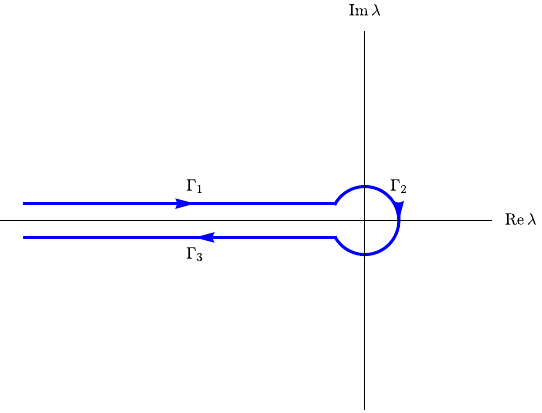}
\caption{Contour of integration.}
\label{contour}
\end{figure}

\

The classical theory of complex powers of elliptic pseudodifferential operators (see, e.g., \cite{seeley}, \cite[\S~10]{shubin}) tells us that
under the assumption \eqref{s is greater than -1}
\begin{equation}
\label{complex power of Delta to appropriate power}
(-\boldsymbol{\Delta})^{-\frac{s+1}{2}}=\frac{i}{2\pi}\ointclockwise_\Gamma \lambda^{-\frac{s+1}{2}} R_\lambda\,\dr \lambda\,,
\end{equation}
so that, arguing as in \cite[\S~11.2, formula (11.10)]{shubin}, we have
\begin{equation}
\label{q as complex integral}
[q_{-s-j-1}]_{\alpha}{}^\beta(x,\xi)=\frac{i}{2\pi}\ointclockwise_\Gamma \lambda^{-\frac{s+1}{2}} [r_{-2-j}]_\alpha{}^\beta(x,\xi,\lambda)\,\dr\lambda\,,\qquad j=0,1,2,\dots,
\end{equation}
where
\begin{equation}
\label{symbol Rlambda}
r_\alpha{}^\beta(x,\xi,\lambda)\sim \sum_{j=0}^{+\infty} [r_{-2-j}]_\alpha{}^\beta(x,\xi,\lambda), \qquad [r_{-2-j}]_\alpha{}^\beta(x,t \xi,t^2\lambda)=t^{-2-j}[r_{-2-j}]_\alpha{}^\beta(x,\xi,\lambda)\quad \forall t>0\,,
\end{equation}
is the full symbol of $R_\lambda$ as a pseudodifferential operator depending on the parameter $\lambda$. Here the branch of $\lambda^{-\frac{s+1}{2}}$ is determined by making a cut along the negative real semi-axis and requiring that $\lambda^{-\frac{s+1}{2}}=1$ when $\lambda=1$.

\

Formula \eqref{q as complex integral} reduces the the analysis of the symbol of $(-\boldsymbol{\Delta})^{-\frac{s+1}{2}}$ to the analysis of the symbol of $R_\lambda\,$.

\

It is easy to see that
\begin{equation}
\label{principal symbol R lambda}
[(R_\lambda)_\prin]_\alpha{}^\beta (x,\xi,\lambda)=[r_{-2}]_{\alpha}{}^\beta(x,\xi,\lambda)=\frac{1}{\|\xi\|^2-\lambda}\,\delta_\alpha{}^\beta\,.
\end{equation}
Note that $(R_\lambda)_\prin(x,\xi,\lambda)$ is holomorphic in $\{|\lambda|\le \rho\}\subset \mathbb{C}$.

\

The following is an auxiliary lemma which will prove to be useful in subsequent calculations.

\begin{lemma}
\label{lemma residues}
We have 
\begin{equation*}
\frac{i}{2\pi}\ointclockwise_{\Gamma}\frac{\lambda^{-\frac{s+1}{2}}}{(\|\xi\|^2-\lambda)^n}\,\dr \lambda= C_n(s)\,\|\xi\|^{-s-2n+1}\,,
\end{equation*}
where
\begin{equation*}
C_n(s):=
\begin{cases}
1 &\text{for}\quad n=1,\\
\frac{s+1}{2} &\text{for}\quad n=2,\\
 \frac{(s+1)(s+3)}{8} &\text{for}\quad n=3,\\
\frac{(s+1)(s+3)(s+5)}{48}&\text{for}\quad n=4,\\
 \frac{(s+1)(s+3)(s+5)(s+7)}{384}&\text{for}\quad n=5.\\
\end{cases}
\end{equation*}
\end{lemma}

\begin{proof}
The claim is a straightforward consequence of Cauchy's residue theorem.
\end{proof}

We are now in a position to state the main result of this subsection.

\begin{proposition}
\label{proposition As is of order -s-3}
The parameter-dependent asymmetry operator $A^{(s)}$ is a pseudodifferential operator of order $-s-3$.
\end{proposition}

The proof of Proposition~\ref{proposition As is of order -s-3} consists of a careful examination of the symbol of the operator \eqref{alternative representation for As} in geodesic normal coordinates by means of formulae \eqref{symbol q of Delta -s-1/2}, \eqref{symbol curl Delta -s-1/2}, Lemma~\ref{lemma Apt}, formulae \eqref{complex power of Delta to appropriate power}--\eqref{symbol Rlambda}, and the results from Appendix~\ref{Symbol of the Hodge Laplacian}. To keep the paper to a reasonable length, we decided to omit this rather long and technical proof, on the basis that it is very similar in spirit to the proof of part (b) of Theorem~\ref{Theorem order of A(s)}, which we will give in full in the next subsection. The latter will provide a `quantitative version' of the arguments required to prove Proposition~\ref{proposition As is of order -s-3}.

\subsection{The principal symbol of $A^{(s)}$}
\label{The principal symbol of As}

The goal of this subsection is to prove part (b) of Theorem~\ref{Theorem order of A(s)}. The precise expression for the principal symbol of $A^{(s)}$ will play a crucial role in the proof of the main result of our paper. This warrants providing a detailed proof.

\begin{proof}[Proof of Theorem~\ref{Theorem order of A(s)}, part (b)]
As in the previous subsection, in what follows we assume to have chosen geodesic normal coordinates centred at $z=0$. In view of Lemma~\ref{lemma Apt}, proving \eqref{Theorem order of A(s) equation 1} reduces to showing that, in the chosen coordinate system, we have
\begin{equation}
\label{proof As prin equation 1}
[t_{-s-3}]_\alpha{}^\alpha(0,\xi)
=
-\frac{(s+1)(s+3)}{6\,|\xi|^{s+5}}\,
\varepsilon^{\alpha\beta \gamma}\,
\nabla_\alpha \operatorname{Ric}_{\beta}{}^\rho(0)\, \xi_\gamma\xi_\rho
\,,
\end{equation}
where $\,t_{\alpha}{}^\beta(x,\xi)\,$ is defined in accordance with
\eqref{symbol curl Delta -s-1/2}
and
\eqref{symbol q of Delta -s-1/2}.

It is not hard to convince oneself that $A^{(s)}$ should be proportional to the covariant derivative of the Ricci tensor and the totally antisymmetric tensor $E$, see, e.g., \cite[Remark~6.8]{curl}. Therefore, without loss of generality one can assume that the metric admits the following Taylor expansion in geodesic normal coordinates:
\begin{equation}
\label{16 November 2023 equation 1}
g_{\alpha\beta}(x)
=
\delta_{\alpha\beta}
-
\frac16
(\nabla_\sigma \Riem_{\alpha\mu\beta\nu})(0)\,x^\sigma\,x^\mu\,x^\nu
+O(|x|^4).
\end{equation}
Namely, one can assume that all components of curvature (but not their covariant derivatives) vanish at the centre of the normal coordinate system $z=0$.
Formula \eqref{16 November 2023 equation 1} immediately implies 
%\begin{equation}
%\label{16 November 2023 equation 2}
%\det g_{\alpha\beta}=1-\frac16 
%(\nabla_\sigma \Ric_{\mu\nu})(0)\,x^\sigma\,x^\mu\,x^\nu
%+O(|x|^4),
%\end{equation}
\begin{equation*}
\label{16 November 2023 equation 3}
\rho(x)=1-\frac1{12} 
(\nabla_\sigma \Ric_{\mu\nu})(0)\,x^\sigma\,x^\mu\,x^\nu
+O(|x|^4).
\end{equation*}

Under this assumption, the symbol of the Hodge Laplacian admits the expansion given in Theorem~\ref{theorem symbol of the Hodge Laplacian} with $a_1=0$ and $a_0=0$.

To start with, let us note the following expressions for the
partial derivatives in momentum $\xi$ of $(R_\lambda)_\prin$ \eqref{principal symbol R lambda}
of order $\,\le 3\,$:
\begin{subequations}
\label{derivatives in xi of Rprin}
\begin{equation}
\label{derivatives in xi of Rprin eq 1}
[(\|\xi\|^2-\lambda)^{-1}]_{\xi_\mu}=-\frac{2g^{\mu\sigma}\xi_\sigma}{(\|\xi\|^2-\lambda)^{2}}\,,
\end{equation}
\begin{equation}
\label{derivatives in xi of Rprin eq 2}
 [(\|\xi\|^2-\lambda)^{-1}]_{\xi_\mu\xi_\nu}=-\frac{2g^{\mu\nu}}{(\|\xi\|^2-\lambda)^{2}}+\frac{8g^{\mu\sigma}g^{\nu\rho}\xi_\sigma\xi_\rho}{(\|\xi\|^2-\lambda)^{3}}\,,
\end{equation}
\begin{equation}
\label{derivatives in xi of Rprin eq 3}
  [(\|\xi\|^2-\lambda)^{-1}]_{\xi_\mu\xi_\nu\xi_\kappa}=8\frac{g^{\mu\nu}\xi^\kappa+g^{\mu\kappa}\xi^\nu+g^{\nu\kappa}\xi^\mu}{(\|\xi\|^2-\lambda)^{3}}-\frac{48 \xi^\mu\xi^\nu\xi^\kappa}{(\|\xi\|^2-\lambda)^{4}}\,.
\end{equation}
\end{subequations}
In formula \eqref{derivatives in xi of Rprin eq 3} and some subsequent formulae we employ,
for the sake of brevity, the notation
\begin{equation*}
\label{xi with raised tensor index}
\xi^\alpha:=g^{\alpha\beta}\xi_\beta\,.
\end{equation*}

\

Using formulae \eqref{5 November 2021 equation 1} and \eqref{principal symbol R lambda} one obtains
\begin{equation*}
[(R_\lambda (-\boldsymbol{\Delta}-\lambda))_{-1}]_\alpha{}^\beta=
\frac{1}{\|\xi\|^2-\lambda}[h_1]_\alpha{}^\beta+(\|\xi\|^2-\lambda)[r_{-3}]_\alpha{}^\beta\\-i [(\|\xi\|^2-\lambda)^{-1}]_{\xi_\mu}(\|\xi\|^2-\lambda)_{x^\mu}\,,
\end{equation*}
where $[h_1]_\alpha{}^\beta$ is defined in accordance with
\eqref{10 August 2023 equation 16},
\eqref{theorem symbol of the Hodge Laplacian equation 3bis}
and
\eqref{10 August 2023 equation 18}
and $[r_{-3}]_\alpha{}^\beta$ comes from \eqref{symbol Rlambda}.
Formula~\eqref{derivatives in xi of Rprin eq 1} and the identity $R_\lambda (-\boldsymbol{\Delta}-\lambda)=\mathrm{Id}$ then imply
\begin{equation*}
\label{16 November 2023 equation 6}
[r_{-3}]_\alpha{}^\beta=-\frac{1}{(\|\xi\|^2-\lambda)^2}[h_1]_\alpha{}^\beta- \frac{2i}{(\|\xi\|^2-\lambda)^{3}} g^{\mu\gamma}\xi_\gamma (\|\xi\|^2)_{x^\mu}\,\delta_\alpha{}^\beta\,.
\end{equation*}
The latter, combined with \eqref{q as complex integral} and Lemma~\ref{lemma residues}, in turn gives us
\begin{equation}
\label{16 November 2023 equation 7}
[q_{-s-2}]_\alpha{}^\beta=-\frac{s+1}{2\|\xi\|^{s+3}}[h_1]_\alpha{}^\beta- \frac{i(s+1)(s+3)}{4\|\xi\|^{s+5}} g^{\mu\gamma}\xi_\gamma (\|\xi\|^2)_{x^\mu}\,\delta_\alpha{}^\beta\,,
\end{equation}
where $[q_{-s-2}]_\alpha{}^\beta$ comes from \eqref{symbol q of Delta -s-1/2}.

Similar arguments give us
\begin{multline}
\label{16 November 2023 equation 8}
[r_{-4}]_\alpha{}^\beta
=
-\frac{1}{(\|\xi\|^2-\lambda)^2}[h_0]_\alpha{}^\beta
\\
-\frac1{(\|\xi\|^2-\lambda)^{3}}
\left[ 
2ig^{\mu\sigma}\xi_\sigma ([h_1]_\alpha{}^\beta)_{x^\mu} +g^{\mu\nu}(\|\xi\|^2)_{x^\mu x^\nu}\delta_\alpha{}^\beta
\right]
\\
+\frac{4g^{\mu\sigma}g^{\nu\rho}\xi_\sigma\xi_\rho}{(\|\xi\|^2-\lambda)^{4}}(\|\xi\|^2)_{x^\mu x^\nu}\delta_\alpha{}^\beta
+ O(|x|^2|\xi|^{-4})\,,
\end{multline}
\begin{multline}
\label{16 November 2023 equation 9}
[q_{-s-3}]_\alpha{}^\beta
=
-\frac{s+1}{2\|\xi\|^{s+3}}[h_0]_\alpha{}^\beta
\\
-\frac{(s+1)(s+3)}{8\|\xi\|^{s+5}}
\left[ 
2ig^{\mu\sigma}\xi_\sigma ([h_1]_\alpha{}^\beta)_{x^\mu} +g^{\mu\nu}(\|\xi\|^2)_{x^\mu x^\nu}\delta_\alpha{}^\beta
\right]
\\
+\frac{(s+1)(s+3)(s+5)g^{\mu\sigma}g^{\nu\rho}\xi_\sigma\xi_\rho}{12\|\xi\|^{s+7}}(\|\xi\|^2)_{x^\mu x^\nu}\delta_\alpha{}^\beta
+ O(|x|^2|\xi|^{-s-3})\,,
\end{multline}
\begin{multline}
\label{16 November 2023 equation 13}
[r_{-5}]_\alpha{}^\beta
=
\frac1{(\|\xi\|^2-\lambda)^{3}} \left(-2i\xi^\mu ([h_0]_\alpha{}^\beta)_{x^\mu} -g^{\mu\nu} ([h_1]_\alpha{}^\beta)_{x^\mu x^\nu}\right)
\\
+\frac1{(\|\xi\|^2-\lambda)^{4}}\left(4\xi^\mu\xi^\nu([h_1]_\alpha{}^\beta)_{x^\mu x^\nu}-\frac{8i}{3}(g^{\mu\nu}\xi^\sigma+g^{\mu\sigma}\xi^\nu+g^{\nu\sigma}\xi^\mu)(\|\xi\|^2)_{x^\mu x^\nu x^\sigma}\,\delta_{\alpha}{}^\beta\right)\,
\\
+\frac{16i \xi^\mu\xi^\nu\xi^\sigma}{(\|\xi\|^2-\lambda)^{5}}(\|\xi\|^2)_{x^\mu x^\nu x^\sigma}\,\delta_{\alpha}{}^\beta
+O(|x||\xi|^{-5})\,,
\end{multline}
\begin{multline}
\label{16 November 2023 equation 14}
[q_{-s-4}]_\alpha{}^\beta
=
\frac{(s+1)(s+3)}{8|\xi|^{s+5}} \left(-2i\xi^\mu ([h_0]_\alpha{}^\beta)_{x^\mu} -g^{\mu\nu} ([h_1]_\alpha{}^\beta)_{x^\mu x^\nu}\right)
\\
+\frac{(s+1)(s+3)(s+5)}{48|\xi|^{s+7}}\left(4\xi^\mu\xi^\nu([h_1]_\alpha{}^\beta)_{x^\mu x^\nu}-\frac{8i}{3}(g^{\mu\nu}\xi^\sigma+g^{\mu\sigma}\xi^\nu+g^{\nu\sigma}\xi^\mu)(\|\xi\|^2)_{x^\mu x^\nu x^\sigma}\,\delta_{\alpha}{}^\beta\right)\,
\\
+\frac{i(s+1)(s+3)(s+5)(s+7)\xi^\mu\xi^\nu\xi^\sigma}{24|\xi|^{s+9}}(\|\xi\|^2)_{x^\mu x^\nu x^\sigma}\,\delta_{\alpha}{}^\beta
+O(|x||\xi|^{-s-4})\,.
\end{multline}
To simplify \eqref{16 November 2023 equation 8}--\eqref{16 November 2023 equation 14} we used the facts that
\begin{equation*}
\label{estimates symbol Hodge Laplacian}
h_{1}=O(|x|^2|\xi|),\qquad h_{0}=O(|x|)\,,
\end{equation*}
which follow from \eqref{16 November 2023 equation 1} and Theorem~\ref{theorem symbol of the Hodge Laplacian}. Observe also that
\begin{equation*}
\label{estimates components of r}
r_{-3}=O(|x|^2|\xi|^{-3}), \qquad r_{-4}=O(|x||\xi|^{-4})\,.
\end{equation*}

\

We now just need to put the various pieces together. Indeed, \eqref{symbol curl Delta -s-1/2} tells us that
\begin{equation}
\label{16 November 2023 equation 15}
[t_{-s-3}]_\alpha{}^\alpha(0,\xi) = \varepsilon_\alpha{}^{\mu\gamma}\left(i \xi_\mu\, [q_{-s-4}]_{\gamma}{}^\alpha(0,\xi) + \frac{\partial [q_{-s-3}]_{\gamma}{}^\alpha}{\partial x^\mu}(0,\xi) \right)\,.
\end{equation}

\

Differentiating \eqref{16 November 2023 equation 9} with respect to $x$ and relabelling the indices we get
\begin{multline}
\label{16 November 2023 equation 11}
([q_{-s-3}]_\gamma{}^\alpha)_{x^\mu}(x,\xi)
=
-\underset{(*)}{\underbrace{\frac{s+1}{2|\xi|^{s+3}}([h_0]_\gamma{}^\alpha)_{x^\mu}}}
\\
-\frac{(s+1)(s+3)}{8|\xi|^{s+5}}
\left[ 
2i \xi^\rho ([h_1]_\gamma{}^\alpha)_{x^\rho x^\mu} +\underset{(*)}{\underbrace{g^{\rho\sigma}(\|\xi\|^2)_{x^\rho x^\sigma x^\mu}\delta_\gamma{}^\alpha}}
\right]
\\
+\underset{(*)}{\underbrace{\frac{(s+1)(s+3)(s+5) \xi^\rho\xi^\sigma}{12|\xi|^{s+7}}(\|\xi\|^2)_{x^\rho x^\sigma x^\mu}\delta_\gamma{}^\alpha}}
+ O(|x||\xi|^{-s-3})\,.
\end{multline}
For convenience, let us also relabel the indices in \eqref{16 November 2023 equation 14} to obtain
\begin{multline}
\label{23 November 2023 equation 8}
[q_{-s-4}]_\gamma{}^\alpha
=
\frac{(s+1)(s+3)}{8|\xi|^{s+5}} \left(-2i\xi^\rho ([h_0]_\gamma{}^\alpha)_{x^\rho} -g^{\rho\sigma} ([h_1]_\gamma{}^\alpha)_{x^\rho x^\sigma}\right)
\\
+\frac{(s+1)(s+3)(s+5)}{48|\xi|^{s+7}}\left(4\xi^\rho\xi^\sigma([h_1]_\gamma{}^\alpha)_{x^\rho x^\sigma}-\underset{(*)}{\underbrace{\frac{8i}{3}(g^{\mu\nu}\xi^\sigma+g^{\mu\sigma}\xi^\nu+g^{\nu\sigma}\xi^\mu)(\|\xi\|^2)_{x^\mu x^\nu x^\sigma}\,\delta_{\gamma}{}^\alpha}}\right)\,
\\
+\underset{(*)}{\underbrace{\frac{i(s+1)(s+3)(s+5)(s+7)\xi^\mu\xi^\nu\xi^\sigma}{24|\xi|^{s+9}}(\|\xi\|^2)_{x^\mu x^\nu x^\sigma}\,\delta_{\gamma}{}^\alpha}}
+O(|x||\xi|^{-s-4})\,.
\end{multline}

It is not hard to see that 
the terms marked by (*) in \eqref{16 November 2023 equation 11} and \eqref{23 November 2023 equation 8} vanish when substituted into \eqref{16 November 2023 equation 15}.
This happens because in each case we are looking at a contraction
of the totally antisymmetric symbol $\varepsilon$
with a rank 3 tensor symmetric in a pair of indices.

\

Let us now substitute \eqref{16 November 2023 equation 11} and \eqref{23 November 2023 equation 8} into \eqref{16 November 2023 equation 15}, drop the terms marked by $(*)$, and examine the
contribution from the remaining term proportional to
\begin{equation}
\label{dima proportionality coefficient 1}
\frac{(s+1)(s+3)(s+5)}{|\xi|^{s+7}}\,.
\end{equation}
Dropping the factor \eqref{dima proportionality coefficient 1},
we have
\begin{multline}
\label{23 November 2023 equation 9}
\left.\frac{i}{12}\,\varepsilon_\alpha{}^{\mu\gamma} \xi_\mu\xi^\rho\xi^\sigma([h_1]_\gamma{}^\alpha)_{x^\rho x^\sigma}\right|_{x=0}
\\
=
\frac{1}{72}\varepsilon_\alpha{}^{\mu\gamma} \xi_\mu \xi_\tau \xi^\rho\xi^\sigma
\left[ 
\nabla_\gamma\operatorname{Riem}^\tau{}_\rho{}^\alpha{}_\sigma(0)
-3
\nabla_\rho\operatorname{Riem}^\tau{}_\gamma{}^\alpha{}_\sigma(0)
+5
\nabla_\sigma\operatorname{Riem}^\tau{}_\rho{}^\alpha{}_\gamma(0)
\right.
\\
\left.
+
\nabla_\gamma\operatorname{Riem}^\tau{}_\sigma{}^\alpha{}_\rho(0)
-3
\nabla_\sigma\operatorname{Riem}^\tau{}_\gamma{}^\alpha{}_\rho(0)
+5
\nabla_\rho\operatorname{Riem}^\tau{}_\sigma{}^\alpha{}_\gamma(0)
\right]
\\
=
\frac{1}{72} \varepsilon_\alpha{}^{\mu\gamma} \xi_\mu \xi_\tau \xi^\rho\xi^\sigma
\left[ 
-3
\nabla_\rho\operatorname{Riem}^\tau{}_\gamma{}^\alpha{}_\sigma(0)
-3
\nabla_\sigma\operatorname{Riem}^\tau{}_\gamma{}^\alpha{}_\rho(0)
\right]
\\
=
-\frac{1}{12}
\varepsilon_\alpha{}^{\mu\gamma} \xi_\mu \xi_\tau \xi^\rho\xi^\sigma
\nabla_\rho\operatorname{Riem}^\tau{}_\gamma{}^\alpha{}_\sigma(0)
=
\frac{1}{12}
\varepsilon^{\mu\alpha\gamma}
\xi_\mu\xi^\rho
\xi^\tau\xi^\sigma
\nabla_\rho\operatorname{Riem}_{\tau\gamma\alpha\sigma}(0)
\\
=
\frac{1}{24}
\varepsilon^{\mu\alpha\gamma}
\xi_\mu\xi^\rho
\xi^\tau\xi^\sigma
\left[
\nabla_\rho\operatorname{Riem}_{\tau\gamma\alpha\sigma}(0)
-
\nabla_\rho\operatorname{Riem}_{\tau\alpha\gamma\sigma}(0)
\right]
\\
=
\frac{1}{48}
\varepsilon^{\mu\alpha\gamma}
\xi_\mu\xi^\rho
\xi^\tau\xi^\sigma
\left[
\nabla_\rho\operatorname{Riem}_{\tau\gamma\alpha\sigma}(0)
-
\nabla_\rho\operatorname{Riem}_{\tau\alpha\gamma\sigma}(0)
+
\nabla_\rho\operatorname{Riem}_{\sigma\gamma\alpha\tau}(0)
-
\nabla_\rho\operatorname{Riem}_{\sigma\alpha\gamma\tau}(0)
\right]
\\
=0\,.
\end{multline}
Hence, the term from \eqref{23 November 2023 equation 8}
proportional to \eqref{dima proportionality coefficient 1}
does not contribute to \eqref{16 November 2023 equation 15}.

We claim that the only surviving term in \eqref{16 November 2023 equation 11} does not contribute to \eqref{16 November 2023 equation 15} either. Indeed, dropping the factor
$(s+1)(s+3)|\xi|^{-s-5}\,$,
we have
\begin{multline}
\label{23 November 2023 equation 11}
\left.-\frac{i}{4}\,\varepsilon_\alpha{}^{\mu\gamma}
\xi^\rho ([h_1]_\gamma{}^\alpha)_{x^\rho x^\mu}\right|_{x=0}
\\
=
-\frac1{24}
\varepsilon_\alpha{}^{\mu\gamma}\, \xi^\rho\xi_\tau\left[
\nabla_\gamma\operatorname{Riem}^\tau{}_\mu{}^\alpha{}_\rho(0)
-3
\nabla_\mu\operatorname{Riem}^\tau{}_\gamma{}^\alpha{}_\rho(0)
+5
\nabla_\rho\operatorname{Riem}^\tau{}_\mu{}^\alpha{}_\gamma(0)
\right.
\\
\left.
+\nabla_\gamma\operatorname{Riem}^\tau{}_\rho{}^\alpha{}_\mu(0)
-3
\nabla_\rho\operatorname{Riem}^\tau{}_\gamma{}^\alpha{}_\mu(0)
+5
\nabla_\mu\operatorname{Riem}^\tau{}_\rho{}^\alpha{}_\gamma(0)
 \right]
\\
=
-\frac1{24}
\varepsilon_\alpha{}^{\mu\gamma}\, \xi^\rho\xi_\tau\left[
\nabla_\gamma\operatorname{Riem}^\tau{}_\mu{}^\alpha{}_\rho(0)
-3
\nabla_\mu\operatorname{Riem}^\tau{}_\gamma{}^\alpha{}_\rho(0)
+5
\nabla_\rho\operatorname{Riem}^\tau{}_\mu{}^\alpha{}_\gamma(0)
-3
\nabla_\rho\operatorname{Riem}^\tau{}_\gamma{}^\alpha{}_\mu(0)
 \right]
\\
=
-\frac1{24}
\varepsilon_\alpha{}^{\mu\gamma}\, \xi^\rho\xi_\tau
\left[
-4
\nabla_\mu\operatorname{Riem}^\tau{}_\gamma{}^\alpha{}_\rho(0)
+8
\nabla_\rho\operatorname{Riem}^\tau{}_\mu{}^\alpha{}_\gamma(0)
\right]
\\
=
-\frac1{3}
\varepsilon_\alpha{}^{\mu\gamma}\, \xi^\rho\xi_\tau
\nabla_\rho\operatorname{Riem}^\tau{}_\mu{}^\alpha{}_\gamma(0)
=
\frac1{3}
\varepsilon^{\mu\alpha\gamma}\,\xi^\rho\xi^\tau
\nabla_\rho\operatorname{Riem}_{\tau\mu\alpha\gamma}(0)
\\
=
\frac1{3}
\xi^\rho\xi^\tau
\left.
\left[
\nabla_\rho
\left(
E^{\mu\alpha\gamma}
\operatorname{Riem}_{\tau\mu\alpha\gamma}
\right)
\right]
\right|_{x=0}
\,.
\end{multline}
But the first (algebraic) Bianchi identity tells us that
the 1-form $\,E^{\mu\alpha\gamma}
\operatorname{Riem}_{\tau\mu\alpha\gamma}\,$ is identically zero, hence,
the quantity \eqref{23 November 2023 equation 11} vanishes.
Alternatively, one can avoid the use of the Bianchi identity 
by expressing the Riemann tensor in terms of the Ricci tensor
according to \cite[formula (6.24)]{curl}.

We have established that the only term contributing to
\eqref{16 November 2023 equation 15}
is the first term from the right-hand side of 
\eqref{23 November 2023 equation 8}.
Thus, working in geodesic normal coordinates
and assuming that curvature vanishes at the origin,
we have arrived at the formula
\begin{multline}
\label{23 November 2023 equation 10}
(A^{(s)})_\mathrm{prin}(0,\xi)
=
[t_{-s-3}]_\alpha{}^\alpha(0,\xi)
\\
=
\frac{(s+1)(s+3)}{8|\xi|^{s+5}}
\varepsilon_\alpha{}^{\mu\gamma}
\left.\left[
2 \xi_\mu\xi^\rho ([h_0]_\gamma{}^\alpha)_{x^\rho} 
-
i \xi_\mu g^{\rho\sigma} ([h_1]_\gamma{}^\alpha)_{x^\rho x^\sigma}
\right]\right|_{x=0}.
\end{multline}
Here
$[h_0]_\gamma{}^\alpha$ is defined in accordance with
\eqref{10 August 2023 equation 17},
\eqref{theorem symbol of the Hodge Laplacian equation 3}
and
\eqref{10 August 2023 equation 19},
whereas
$[h_1]_\gamma{}^\alpha$ is defined in accordance with
\eqref{10 August 2023 equation 16},
\eqref{theorem symbol of the Hodge Laplacian equation 3bis}
and
\eqref{10 August 2023 equation 18}.

Straightforward calculations give us
\begin{equation}
\label{23 November 2023 equation 12}
\left.-i\,\varepsilon_\alpha{}^{\mu\gamma}  \xi_\mu \,g^{\rho\sigma} ([h_1]_\gamma{}^\alpha)_{x^\rho x^\sigma}\right|_{x=0}
=
-\frac83
 \varepsilon^{\alpha\beta\gamma}
\nabla_\alpha\operatorname{Ric}_\beta{}^\rho(0)\,
 \xi_\gamma\xi_\rho
\end{equation}
and
\begin{equation}
\label{24 November 2023 equation 1}
\left.2 \varepsilon_\alpha{}^{\mu\gamma}
 \xi_\mu\xi^\rho ([h_0]_\gamma{}^\alpha)_{x^\rho} \right|_{x=0}
=
\frac43
 \varepsilon^{\alpha\beta\gamma}
\nabla_\alpha\operatorname{Ric}_\beta{}^\rho(0)\,
 \xi_\gamma\xi_\rho\,.
\end{equation}
Substituting \eqref{23 November 2023 equation 12} and \eqref{24 November 2023 equation 1} into \eqref{23 November 2023 equation 10} we arrive at \eqref{proof As prin equation 1}. This concludes the proof of Theorem~\ref{Theorem order of A(s)}, part (b).
\end{proof}

\section{From $A^{(s)}$ to $\eta_{\operatorname{curl}}(s)$}
\label{From As to eta}

Let us make the relation between $A^{(s)}$ and $\eta_{\operatorname{curl}}(s)$ more explicit.

\

Resorting to the spectral representation \eqref{spectral decomposition Hodge Laplacian} for the Hodge Laplacian we get
\begin{equation}
\label{From As to eta equation 1}
\curl\, (-\boldsymbol{\Delta})^{-\frac{s+1}2}
=
\sum_{j\in\mathbb{Z}\setminus\{0\}}\lambda^{-s}\,u_j\langle u_j,\,\cdot\,\rangle\,.
\end{equation}
But then, for $s>3$, \eqref{From As to eta equation 1} and \eqref{alternative representation for As} immediately imply that
\begin{equation}
\label{From As to eta equation 2}
\mathfrak{a}^{(s)}(x,x)=\eta_{\curl}^\mathrm{loc}(x;s)\,, \qquad \int_M \mathfrak{a}^{(s)}(x,x) \dr x=\eta_{\curl}(s)\,.
\end{equation}
The condition $s>3$ ensures that the series expansions over the eigensystem
of $\curl$ converge absolutely.

Formula \eqref{From As to eta equation 2} tells us that a good understanding of the behaviour of $\mathfrak{a}^{(s)}(x,y)$ as a function of~$s$ would enable us to prove our main result.
In particular, in order to go to the limit as $s\to0^+$ we need an explicit description of the leading singularity of the parameter-dependent asymmetry operator $A^{(s)}$.

\section{The leading singularity of the integral kernel of $A^{(s)}$}
\label{The leading singularity of the integral kernel of As}
%
%\section{The reference operator}
%\label{The reference operator}

In this section we define a reference integral operator which captures the leading singularity of the parameter-dependent asymmetry operator $A^{(s)}$.

Consider the exponential map $\exp_x:T_xM\to M$.
In a neigbourhood of $x$ the exponential map has an inverse
$\exp_x^{-1}:M\to T_xM$.
Of course, the inverse exponential map can be expressed in terms of the distance function as
\begin{equation}
\label{The reference operator equation 1}
[\exp_x^{-1}(y)]^\alpha=-\frac{1}{2}\,g^{\alpha\beta}(x)\,
[\operatorname{dist}^2(x,y)]_{x^\beta}\,.
\end{equation}

Put
\begin{equation}
\label{The reference operator equation 2}
\mathfrak{ref}^{(s)}(x,y)
:=
[\operatorname{dist}(x,y)]^{s-2}\,
E^{\alpha\beta}{}_\gamma(x)\,
\nabla_\alpha\operatorname{Ric}_{\beta\sigma}(x)\,
[\exp^{-1}_x(y)]^\gamma\,[\exp^{-1}_x(y)]^\sigma
\,\chi(\operatorname{dist}(x,y)/\epsilon)
\end{equation}
for $x\ne y$ and $\mathfrak{ref}^{(s)}(x,y):=0$ for $x=y$.
Here $s\in(-3,+\infty)$ is a parameter,
$\chi:[0,+\infty)\to \mathbb{R}$ is a compactly supported infinitely smooth scalar function such that $\chi=1$ in a neighbourhood of zero
and $\epsilon$ is a small positive number which ensures that
$\mathfrak{ref}^{(s)}(x,y)$
vanishes when $x$ and $y$ are not sufficiently close.

\begin{definition}
\label{definition of the reference operator}
We call the scalar integral operator
\begin{equation}
\label{The reference operator equation 3}
\operatorname{Ref}^{(s)}
:=\int_M\mathfrak{ref}^{(s)}(x,y)\,(\,\cdot\,)\,\rho(y)\,\dr y
\end{equation}
the \emph{reference operator}.
\end{definition}

\begin{theorem}
\label{principal symbol of the reference operator}
The reference operator is a pseudodifferential operator of order $-s-3$ with principal symbol
\begin{equation}
\label{The reference operator equation 4}
(\operatorname{Ref}^{(s)})_\prin(x,\xi)
=
\,c(s)\,
\|\xi\|^{-s-5}\,
E^{\alpha\beta\gamma}(x)\,
\nabla_\alpha \operatorname{Ric}_{\beta}{}^\sigma(x)\, \xi_\gamma\xi_\sigma\,,
\end{equation}
where
\begin{equation}
\label{The reference operator equation 5}
c(s)
:=
\begin{cases}
-\frac{4\pi\,\Gamma(s+4)}{s(s+2)}\,\sin\frac{\pi s}{2}
\quad&\text{if}\quad s\ne0,-2,
\\
-6\pi^2
\quad&\text{if}\quad s=0,
\\
-\pi^2
\quad&\text{if}\quad s=-2.
\end{cases}
\end{equation}
\end{theorem}

\begin{proof}
Let us fix a point $x\in M$.
The scalar function
\eqref{The reference operator equation 2}
defines a distribution
\begin{equation}
\label{distribution}
\mathfrak{ref}^{(s)}(x,y):
f(y)
\mapsto
\int
\mathfrak{ref}^{(s)}(x,y)\,f(y)\,\rho(y)\,\dr y
\end{equation}
in the variable $y$.
This distribution depends on the parameter $x$.

In what follows we use the same notation $\,\mathfrak{ref}^{(s)}(x,y)\,$
for the scalar function and the distribution.
The meaning will be clear from the context.

Let us choose geodesic normal coordinates $\tilde y$ centred at $x$.
Then for $x=0$ and $\tilde y\ne 0$
the scalar function \eqref{The reference operator equation 2} reads
\begin{equation}
\label{proof principal symbol of the reference operator equation 1}
\mathfrak{ref}^{(s)}(0,\tilde y)=
\varepsilon^{\alpha\beta}{}_\gamma\,
\nabla_\alpha\operatorname{Ric}_{\beta\sigma}(0)\,
\frac{\tilde y^\gamma\tilde y^\sigma}{|\tilde y|^{-s+2}}
\,\chi(|\tilde y|/\epsilon)\,.
\end{equation}
We observe that replacing
\begin{equation}
\label{proof principal symbol of the reference operator equation 1a}
\frac{\tilde y^\gamma\tilde y^\sigma}{|\tilde y|^{-s+2}} \mapsto \frac{\tilde y^\gamma\tilde y^\sigma-\frac13 \delta^{\gamma \sigma}|\tilde y|^2}{|\tilde y|^{-s+2}}
\end{equation}
does not affect the RHS of \eqref{proof principal symbol of the reference operator equation 1}, because $\varepsilon^{\alpha\beta\gamma}\,
\nabla_\alpha\operatorname{Ric}_{\beta\gamma}(0)=0$ due to the symmetries of the geometric quantities involved.

Formulae 
\eqref{proof principal symbol of the reference operator equation 1},
\eqref{proof principal symbol of the reference operator equation 1a}
and
Proposition~\ref{reference function lemma} tell us that for fixed $x\in M$ and in chosen geodesic normal coordinates $\tilde y$ centred at $x$ the distribution
\eqref{distribution}
can be written, modulo $C^\infty$, as
\begin{multline}
\label{proof principal symbol of the reference operator equation 2}
\mathfrak{ref}^{(s)}(0,\tilde y):
f(\tilde y)
\mapsto
(2\pi)^{-3}\,c(s)
\\
\int
e^{-i\tilde y^\tau\tilde\xi_\tau}\,
|\tilde\xi|^{-s-5}\,
\varepsilon^{\alpha\beta\gamma}\,
\nabla_\alpha\operatorname{Ric}_{\beta}{}^\sigma(0)\,\tilde\xi_\gamma\tilde\xi_\sigma\,
(1-\chi(|\tilde\xi|))\,
f(\tilde y)\,\rho(\tilde y)\,\dr\tilde y\,\dr\tilde\xi\,,
\end{multline}
where
\begin{equation}
\label{Riemannian density in normal coordinates}
\rho(\tilde y)=1+O(|\tilde y|^2)
\end{equation}
is the Riemannian density in geodesic normal coordinates.

Let us now switch to arbitrary local coordinates $y\,$.
Then our geodesic normal coordinates $\tilde y$ are expressed via $y$ as
$\tilde y=\tilde y(x,y)\,$, $\tilde y(x,x)=0\,$.
Recall that the choice of geodesic normal coordinates is unique
up to a gauge transformation --- an $x$-dependent 3-dimensional Euclidean rotation.
This gauge can be chosen so that the map $\,x,y\mapsto\tilde y\,$ is smooth.

Let us define the $x$-dependent invertible linear map $\xi\mapsto\tilde\xi$ by imposing the condition
\begin{equation}
\label{proof principal symbol of the reference operator equation 3}
\tilde y^\tau\tilde\xi_\tau
=
(y-x)^\tau\xi_\tau
+
O(|y-x|^2\,\|\xi\|)\,.
\end{equation}
The above condition defines the $x$-dependent invertible linear map $\xi\mapsto\tilde\xi$ uniquely.
Note also that
\begin{equation}
\label{norms of xi and tilde xi are the same}
|\tilde\xi|=\|\xi\|
\end{equation}
because $\|\xi\|$ incorporates the metric tensor at the point $x$.

Put
\begin{equation}
\label{densities 1}
\mu(x,y):=\det(\partial\tilde y^\kappa/\partial y^\lambda)\,,
\qquad
\nu(x):=\det(\partial\tilde\xi_\kappa/\partial\xi_\lambda)\,,
\end{equation}
so that
\begin{equation}
\label{densities 2}
\dr\tilde y=\mu(x,y)\,\dr y\,,
\qquad
\dr\tilde\xi=\nu(x)\,\dr\xi\,.
\end{equation}
Note that formulae
\eqref{proof principal symbol of the reference operator equation 3}
and
\eqref{densities 1}
imply
\begin{equation}
\label{densities 3}
\mu(x,y)\,\nu(x)
=
1
+
O(|y-x|)\,.
\end{equation}

In view of
\eqref{densities 2},
\eqref{proof principal symbol of the reference operator equation 3}
and
\eqref{norms of xi and tilde xi are the same},
the distribution \eqref{proof principal symbol of the reference operator equation 2}
can now be written, modulo $C^\infty$, as
\begin{multline}
\label{proof principal symbol of the reference operator equation 4}
\mathfrak{ref}^{(s)}(x,y):
f(y)
\mapsto
(2\pi)^{-3}\,c(s)
\\
\int
e^{-i\tilde y^\tau\tilde\xi_\tau}\,
\|\xi\|^{-s-5}\,
E^{\alpha\beta\gamma}(x)\,
\nabla_\alpha\operatorname{Ric}_{\beta}{}^\sigma(x)\,\xi_\gamma\xi_\sigma\,
(1-\chi(\|\xi\|))\,
f(y)\,\rho(\tilde y)\,\mu(x,y)\,\nu(x)\,\dr y\,\dr\xi\,,
\end{multline}
where $\tilde y$ is a given function of $x$ and $y$
and $\tilde\xi$ is a given function of $x$ and $\xi$,
the latter being linear in $\xi$.
Combining formulae
\eqref{The reference operator equation 3},
\eqref{proof principal symbol of the reference operator equation 4},
\eqref{Riemannian density in normal coordinates},
\eqref{proof principal symbol of the reference operator equation 3}
and
\eqref{densities 3},
and applying standard microlocal arguments
\cite[\S3~and~\S4]{shubin},
we conclude that our reference operator is a pseudodifferential operator
and that its principal symbol is given by formula
\eqref{The reference operator equation 4}.
\end{proof}

Theorem~\ref{principal symbol of the reference operator}
warrants a number of remarks.

\begin{itemize}
\item
The structure of formulae
\eqref{The reference operator equation 2}
and
\eqref{The reference operator equation 4}
is the same.
It is just a matter of replacing
the inverse exponential map
$\exp_y^{-1}(x)$ with momentum $\xi$,
replacing a power of the distance function with an appropriate power of momentum
and introducing a scaling factor $c(s)$.

\item
We have
\begin{equation*}
\label{The reference operator equation 6}
c(0)=\lim_{s\to0}c(s)\,,
\qquad
c(-2)=\lim_{s\to-2}c(s)\,.
\end{equation*}
Furthermore, the function $c:(-3,+\infty)\to\mathbb{R}$ is infinitely smooth.
\item
We have
\begin{equation*}
\label{The reference operator equation 7}
c(2k)=0\,,\qquad k=1,2,\dots,
\end{equation*}
which is not surprising because for positive even values of $s$ the integral kernel
\eqref{The reference operator equation 2} is infinitely smooth.
\item
Examination of formulae
\eqref{Theorem order of A(s) equation 1}
and
\eqref{The reference operator equation 4}
shows that the principal symbols of the operators $A^{(s)}$ and $\operatorname{Ref}^{(s)}$
differ by a scaling factor.
As $c(s)\ne0$ for $s\in(-3,2)$, we can use the reference operator $\operatorname{Ref}^{(s)}$
for describing the leading singularity of the parameter-dependent asymmetry operator $A^{(s)}$
in the range $s\in(-3,2)$.
\end{itemize}

The reference operator possesses the following important property which plays a crucial role in our analysis.

\begin{proposition}
\label{average of integral of the reference operator over a geodesic sphere}
Let $a,b\in\mathbb{R}$ be such that $-2<a<b$. Then we have
\begin{equation}
\label{The reference operator equation 8}
\lim_{r\to0^+}
\frac{1}{4\pi r^2}
\int_{\mathbb{S}_r(x)}
\mathfrak{ref}^{(s)}(x,y)\,\dr S_y
=0\,,
\end{equation}
where the limit is uniform over $s\in[a,b]$ and $x\in M$.
\end{proposition}

Recall that $\mathbb{S}_r(x)=\{y\in M|\dist(x,y)=r\}$ is the sphere of radius $r$ centred at $x$
and $\dr S_y$ is the surface area element on this sphere.

\begin{proof}[Proof of Proposition~\ref{average of integral of the reference operator over a geodesic sphere}]
Let us fix an arbitrary $x\in M$ and work in geodesic normal coordinates $y$ centred at $x$.
In our chosen coordinate system the geodesic sphere is the standard round 3-sphere,
whereas the surface area elements of the geodesic sphere and the standard round sphere differ
by a factor $1+O(r^2)$ with remainder uniform over $x\in M$. This reduces the proof of the
proposition to showing that
\begin{equation}
\label{The reference operator equation 9}
\int_{\mathbb{S}_r}
\varepsilon^{\alpha\beta}{}_\gamma\,
\nabla_\alpha\operatorname{Ric}_{\beta\sigma}(0)\,
y^\gamma\,y^\sigma\,\dr S
=0\,,
\end{equation}
where $y$ are Cartesian coordinates in $\mathbb{R}^3\,$,
$\,\mathbb{S}_r$ is the standard round sphere,
$\dr S$ is its surface area element
and $\varepsilon^{\alpha\beta}{}_\gamma=\varepsilon_{\alpha\beta\gamma}\,$.
But the fact that
$\nabla_\alpha\operatorname{Ric}_{\beta\sigma}=\nabla_\alpha\operatorname{Ric}_{\sigma\beta}$
and the elementary identity
\begin{equation*}
\label{The reference operator equation 10}
\int_{\mathbb{S}_r}
y^\gamma\,y^\sigma
\,\dr S
=
\frac{4\pi r^2}{3}\,\delta^{\gamma\sigma}
\end{equation*}
immediately imply formula \eqref{The reference operator equation 9}.
\end{proof}

\begin{corollary}
\label{average of integral of the reference operator over a geodesic sphere corollary}
Let $a,b\in\mathbb{R}$ be such that $-1<a<b$. Then we have
\begin{equation}
\label{The reference operator equation 8 corollary}
\lim_{r\to0^+}
\frac{1}{4\pi r^2}
\int_{\mathbb{S}_r(x)}
\mathfrak{ref}^{(s)}(y,x)\,\dr S_y
=0\,,
\end{equation}
where the limit is uniform over $s\in[a,b]$ and $x\in M$.
\end{corollary}

Note that in formulae
\eqref{The reference operator equation 8}
and
\eqref{The reference operator equation 8 corollary}
the arguments of $\mathfrak{ref}^{(s)}$ come in opposite order.
Note also that
condition $a>-1$ in
Corollary~\ref{average of integral of the reference operator over a geodesic sphere corollary}
is more restrictive than
condition $a>-2$ in
Proposition~\ref{average of integral of the reference operator over a geodesic sphere}.

\begin{proof}[Proof of Corollary~\ref{average of integral of the reference operator over a geodesic sphere corollary}]
Examination of formula \eqref{The reference operator equation 2} shows that
\begin{equation*}
\mathfrak{ref}^{(s)}(x,y)-\mathfrak{ref}^{(s)}(y,x)
=
O(\operatorname{dist}^{s+1}(x,y))
\end{equation*}
uniformly over $s\in[a,b]$ and $x,y\in M$, and the required result immediately follows.
\end{proof}

The issue with the reference operator defined in accordance with formulae
\eqref{The reference operator equation 1}--\eqref{The reference operator equation 3}
is that it is, generically, not self-adjoint because the variables $x$ and $y$ appear in a non-symmetric
fashion. Hence, in subsequent analysis it is natural to work with its symmetrised version
\begin{equation*}
\label{The reference operator equation 3 symmetrised}
\operatorname{Ref}^{(s)}_\mathrm{sym}
:=
\frac{1}{2}
[
\operatorname{Ref}^{(s)}
+
(
\operatorname{Ref}^{(s)}
)^*
]
=
\int_M
\mathfrak{ref}^{(s)}_\mathrm{sym}(x,y)
\,(\,\cdot\,)\,\rho(y)\,\dr y\,,
\end{equation*}
\begin{equation*}
\label{The reference operator equation 3 symmetrised extra}
\mathfrak{ref}^{(s)}_\mathrm{sym}(x,y)
:=
\frac{1}{2}
[
\mathfrak{ref}^{(s)}(x,y)
+
\mathfrak{ref}^{(s)}(y,x)
]
\,.
\end{equation*}
Of course, the symmetrised reference operator has the same princiapl symbol as the original one,
\begin{equation}
\label{The reference operator equation 4 symmetrised}
(\operatorname{Ref}^{(s)}_\mathrm{sym})_\prin(x,\xi)
=
\,c(s)\,
\|\xi\|^{-s-5}\,
E^{\alpha\beta\gamma}(x)\,
\nabla_\alpha \operatorname{Ric}_{\beta}{}^\sigma(x)\, \xi_\gamma\xi_\sigma\,.
\end{equation}
Furthermore,
Proposition~\ref{average of integral of the reference operator over a geodesic sphere}
and
Corollary~\ref{average of integral of the reference operator over a geodesic sphere corollary}
imply that, given $a,b\in\mathbb{R}$ such that $-1<a<b$, we have
\begin{equation}
\label{The reference operator equation 8 symmetrised}
\lim_{r\to0^+}
\frac{1}{4\pi r^2}
\int_{\mathbb{S}_r(x)}
\mathfrak{ref}^{(s)}_\mathrm{sym}(x,y)\,\dr S_y
=0\,,
\end{equation}
where the limit is uniform over $s\in[a,b]$ and $x\in M$.

\section{Proof of Theorem~\ref{Theorem main result}}
\label{Proof of Theorem main result}

Let $s\in(-3,2)$.
Put
\begin{equation}
\label{definition of modified A(s)}
\tilde A^{(s)}
:=
A^{(s)}
+
\frac{(s+1)(s+3)}{6\,c(s)}
\,
\operatorname{Ref}^{(s)}_\mathrm{sym}\,,
\end{equation}
where the parameter $c(s)$ is given by formula \eqref{The reference operator equation 5}.
The advantage of working with the operator $\tilde A^{(s)}$
rather than with the original parameter-dependent asymmetry operator $A^{(s)}$ is that
$\tilde A^{(s)}$ has lower order.
Namely,
Theorem~\ref{Theorem order of A(s)}
and
formula \eqref{The reference operator equation 4 symmetrised}
imply that
$\tilde A^{(s)}$ is a self-adjoint pseudodifferential operator of order $-4-s$,
which means, in particular, that it is of trace class for $s>-1$.

The integral kernel $\tilde{\mathfrak{a}}^{(s)}(x,y)$ of the operator
$\tilde A^{(s)}$ is continuous as a function of the variables $x,y\in M$ and $s\in(-1,2)$.
Formulae
\eqref{regularised local trace of curl}
and
\eqref{The reference operator equation 8 symmetrised}
imply that
\begin{equation*}
\label{reduction of Proof of Theorem main result 1}
\psi_{\operatorname{curl}}^\loc(x)
=
\tilde{\mathfrak{a}}^{(0)}(x,x)\,.
\end{equation*}
But
\begin{equation*}
\label{reduction of Proof of Theorem main result 2}
\tilde{\mathfrak{a}}^{(0)}(x,x)
=
\lim_{s\to0^+}\mathfrak{a}^{(s)}(x,x)\,,
\end{equation*}
where $\mathfrak{a}^{(s)}(x,y)$ is the integral kernel of
the original parameter-dependent asymmetry operator $A^{(s)}$.
Hence, in order to prove Theorem~\ref{Theorem main result} it is sufficient to prove that
\begin{equation}
\label{reduction of Proof of Theorem main result 3}
\mathfrak{a}^{(s)}(x,x)
=
\eta_{\operatorname{curl}}^\loc(x;s)
\quad\text{for}\quad
s\in(0,1)\,.
\end{equation}

Let us fix an $x\in M$ and examine the behaviour of the quantities
$\mathfrak{a}^{(s)}(x,x)$
and
$\eta_{\operatorname{curl}}^\loc(x;s)$
as functions of the real variable $s$.

\textbf{Observation 1\ }
For $s>3$ we have
$\mathfrak{a}^{(s)}(x,x)
=
\eta_{\operatorname{curl}}^\loc(x;s)$,
see \eqref{From As to eta equation 2}.

\textbf{Observation 2\ }
The function $\mathfrak{a}^{(s)}(x,x)$ is real analytic for $s\in(0,+\infty)$.
Indeed, we are looking at a parameter-dependent integral operator with continuous integral kernel,
an operator constructed explicitly, and in this construction the parameter $s$ appears in
analytic fashion. The full symbol of the pseudodifferential operator $A^{(s)}$ is analytic
in the variable $s$ and the infinitely smooth part of $A^{(s)}$
(the part not described by the symbol)
is analytic in $s$ as well.

Combining Observations 1 and 2 we conclude that
$\mathfrak{a}^{(s)}(x,x)
=
\eta_{\operatorname{curl}}^\loc(x;s)$
for all $s\in(0,+\infty)$.
This implies \eqref{reduction of Proof of Theorem main result 3}
and completes the proof of Theorem~\ref{Theorem main result}.
\qed

\section{Concluding remarks}
\label{Concluding remarks}

Before concluding our paper, let us briefly comment on the potential implications of our results on the study of eta functions more broadly.

The key tool in our analysis was the use of the parameter-dependent asymmetry operator $A^{(s)}$,
an invariantly defined scalar pseudodifferential operator. We established that for $s>0$
our operator $A^{(s)}$ and eta function $\eta_{\operatorname{curl}}(s)$ are related in accordance with Theorem~\ref{theorem trace A(s)}.
Furthermore, as per Remark~\ref{remark about -1 and -3} we have $A^{(-1)}=A^{(-3)}=0$. This makes one wonder whether the eta function itself vanishes at $s=-1$, or $s=-3$, or both.

It appears to be a known fact \cite{Dowker} that for the special case of the Berger sphere
the eta function $\eta_{\operatorname{curl}}(s)$ vanishes at $s=-1$. However, identifying zeros of the eta function for general Riemann manifolds would require very delicate analysis of analytic continuation to negative values of $s$. Regarding $s=-3$ one faces the additional impediment of circumventing the pole that generically exists at $s=-2$.

Transforming the above intuitive arguments into rigorous mathematical analysis is beyond the scope of the current paper.

\section*{Acknowledgements}
\addcontentsline{toc}{section}{Acknowledgements}

MC was supported by EPSRC Fellowship EP/X01021X/1.

\begin{appendices}

\section{Symbol of the Hodge Laplacian}
\label{Symbol of the Hodge Laplacian}

\begin{theorem}
\label{theorem symbol of the Hodge Laplacian}
Let
\begin{equation*}
h_\alpha{}^\beta(x,\xi)
\sim
\|\xi\|^{2}\delta_\alpha{}^\beta
+
[h_1]_\alpha{}^\beta(x,\xi)
+
[h_0]_\alpha{}^\beta(x,\xi)
\end{equation*}
be the (left) symbol of the operator $-\boldsymbol{\Delta}$.
Then in geodesic normal coordinates centred at $x=0$ we have
\begin{equation}
\label{10 August 2023 equation 16}
[h_1]_\alpha{}^\beta(x,\xi)
=
i\left([a_1]_\alpha{}^{\beta\gamma}{}_{\mu}+[b_1]_\alpha{}^{\beta\gamma}{}_{\mu\nu}\,x^\nu\right)\xi_\gamma\,x^\mu
+
O(|\xi|\,|x|^3)\,,
\end{equation}
\begin{equation}
\label{10 August 2023 equation 17}
[h_0]_\alpha{}^\beta(x)
=
[a_0]_\alpha{}^\beta+[b_0]_\alpha{}^\beta{}_\nu\,x^\nu
+
O(|x|^2)\,,
\end{equation}
where
\begin{equation}
\label{theorem symbol of the Hodge Laplacian equation 3}
[a_0]_\alpha{}^\beta
:=
\Ric_{\alpha}{}^\beta(0),
\end{equation}
\begin{equation}
\label{theorem symbol of the Hodge Laplacian equation 3bis}
[a_1]_\alpha{}^{\beta\gamma}{}_{\mu}
:=
\Ric^\gamma{}_\mu(0)\,\delta_{\alpha}{}^\beta
-\frac23
\left(\Riem^\gamma{}_\mu{}^\beta{}_\alpha(0) +\Riem^\beta{}^\gamma{}_\alpha{}_\mu(0)
\right) ,
\end{equation}
\begin{multline}
\label{10 August 2023 equation 18}
[b_1]_\alpha{}^{\beta\gamma}{}_{\mu\nu}
:=
\left[
\frac12\nabla_\mu\operatorname{Ric}^\gamma{}_\nu
-
\frac1{12}\nabla^\gamma\operatorname{Ric}_{\mu\nu}
\right]
\delta_\alpha{}^\beta
\\
-
\frac16
\left[
\nabla_\alpha\operatorname{Riem}^\gamma{}_\mu{}^\beta{}_\nu
-3
\nabla_\mu\operatorname{Riem}^\gamma{}_\alpha{}^\beta{}_\nu
+5
\nabla_\nu\operatorname{Riem}^\gamma{}_\mu{}^\beta{}_\alpha
\right],
\end{multline}
\begin{equation}
\label{10 August 2023 equation 19}
[b_0]_\alpha{}^\beta{}_\nu
:=
-
\frac{1}{6}
\nabla^\beta\operatorname{Ric}_{\alpha\nu}
+
\frac{1}{2}
\nabla_\alpha\operatorname{Ric}^\beta{}_\nu
+
\frac{1}{2}
\nabla_\nu\operatorname{Ric}_\alpha{}^\beta\,.
\end{equation}
\end{theorem}

\begin{remark}
Note that the (left) symbol of $-\boldsymbol{\Delta}$ was computed --- under the special assumption that $\Riem(0)=0$ --- in \cite[Appendix~F]{curl}. Theorem~\ref{theorem symbol of the Hodge Laplacian} agrees with \cite[Lemma~F.2]{curl} when $\Riem(0)=0$. To ease the comparison, let us point out that in \cite{curl} the symbol $h$ was denoted by $q$, the quantity $b_1$ by $a$, and the quantity $b_0$ by $b$.
\end{remark}

\begin{lemma}
\label{lemma subprincipal symbol powers of the Hodge Laplacian}
We have\footnote{Recall that the subprincipal symbol of a pseudodifferential operator acting on 1-forms is defined in accordance with \cite[Definition~3.2]{curl}.}
\begin{equation}
\label{lemma subprincipal symbol powers of the Hodge Laplacian equation 1}
[(-\boldsymbol{\Delta})^{-\frac{s+1}{2}}]_\sub =0\,.
\end{equation}
\end{lemma}

\begin{proof}
Let us fix $z\in M$ and choose geodesic normal coordinates centred at $z=0$. Since the subprincipal symbol is a covariant quantity under changes of local coordinates, it is enough to show that
\begin{equation*}
\label{proof lemma subprincipal symbol powers of the Hodge Laplacian equation 1}
[(-\boldsymbol{\Delta})^{-\frac{s+1}{2}}]_\sub(0,\xi)=0\,.
\end{equation*}
Since, clearly, in the chosen coordinate system we have
\begin{equation*}
\label{proof lemma subprincipal symbol powers of the Hodge Laplacian equation 2}
[(-\boldsymbol{\Delta})^{-\frac{s+1}{2}}]_\prin(x,\xi)=|\xi|^{-\frac{s+1}{2}}\mathrm{I}+O(|x|^2|\xi|^{-\frac{s+1}{2}}),
\end{equation*}
formula \cite[Definition~3.6]{curl} implies
\begin{equation*}
\label{proof lemma subprincipal symbol powers of the Hodge Laplacian equation 3}
[(-\boldsymbol{\Delta})^{-\frac{s+1}{2}}]_\sub(0,\xi)=q_{-s-2}(0,\xi)\,,
\end{equation*}
where we are using notation from
\eqref{symbol q of Delta -s-1/2}.
But formulae \eqref{16 November 2023 equation 7} and \eqref{10 August 2023 equation 16} immediately give us 
\begin{equation}
\label{proof lemma subprincipal symbol powers of the Hodge Laplacian equation 4}
q_{-s-2}(0,\xi)=0\,,
\end{equation}
which concludes the proof.
\end{proof}

\begin{remark}
Of course, formula \eqref{lemma subprincipal symbol powers of the Hodge Laplacian equation 1} can be equivalently rewritten as
\begin{equation*}
[(-\boldsymbol{\Delta})^r]_\sub =0, \qquad r\in \mathbb{R}.
\end{equation*}
\end{remark}

\section{Proof of Lemma~\ref{lemma Apt}}
\label{Proof of Lemma Apt}

The following two lemmata complete the proof of Lemma~\ref{lemma Apt}. Recall that we have fixed a point $z\in M$ and chosen geodesic normal coordinates centred at $z$.

\begin{lemma}
\label{lemma apt -2}
We have
\begin{equation*}
\label{lemma apt -2 equation 1}
(a^{(s)}_\mathrm{pt})_{-s-2}(z,\xi)=0.
\end{equation*}
\end{lemma}
\begin{proof}
Substituting
\eqref{prop: expansion of parallel transoport equation 2}
and
\eqref{aux4}
into
\eqref{As in normal coordinates},
subtracting the diagonal contribution \eqref{As diag}, and integrating by parts, we obtain
%\begin{multline}
\begin{equation}
\label{apt -2 equation 1}
(a^{(s)}_\mathrm{pt})_{-s-2}(z,\xi)=\frac16\, \Riem^{\alpha}{}_{\mu\kappa\nu}(z)\,\dfrac{\partial^2[t_{-s}]_{\alpha}{}^\kappa(z,\xi)}{\partial \xi_\mu \partial \xi_\nu}
%\\
=
\frac16 \,\Riem_{\alpha\mu\kappa\nu}(z)\,\dfrac{\partial^2[i\varepsilon^{\alpha\gamma\kappa} \xi_\gamma\, |\xi|^{-s-1}]}{\partial \xi_\mu \partial \xi_\nu}\,,
\end{equation}
%\end{multline}
where $|\cdot|$ is the Euclidean norm.
By the symmetries of the Riemann tensor we have
\begin{equation}
\label{riemann1}
\Riem_{\alpha\mu\kappa\nu}=\frac12(R_{\alpha\mu\kappa\nu}+\Riem_{\kappa\nu\alpha\mu}).
\end{equation}
Since 
\begin{equation}
\label{riemann1a}
\frac{\partial^2[i\varepsilon^{\alpha\gamma\kappa} \xi_\gamma\, |\xi|^{-s-1}]}{\partial \xi_\mu \partial \xi_\nu}
\end{equation}
is symmetric in the pair of indices $\mu$ and $\nu$, we can replace \eqref{riemann1} in \eqref{apt -2 equation 1} with its symmetrised version
\begin{equation}
\label{riemann2}
\Riem_{\alpha\mu\kappa\nu}=\frac12(\Riem_{\alpha\mu\kappa\nu}+\Riem_{\kappa\nu\alpha\mu}+\Riem_{\alpha\nu\kappa\mu}+\Riem_{\kappa\mu\alpha\nu}).
\end{equation}
But the quantity \eqref{riemann2} is symmetric in the pair of indices $\alpha$ and $\kappa$, whereas \eqref{riemann1a} is antisymmetric in the same pair of indices. Therefore \eqref{apt -2 equation 1} vanishes.
\end{proof}

\begin{lemma}
\label{lemma apt -3}
We have
\begin{equation*}
\label{lemma apt -3 equation 1}
(a^{(s)}_\mathrm{pt})_{-s-3}(z,\xi)=0.
\end{equation*}
\end{lemma}
\begin{proof}
First of all, let us observe that formulae \eqref{proof lemma subprincipal symbol powers of the Hodge Laplacian equation 4}, \eqref{symbol q of Delta -s-1/2}, \eqref{symbol curl Delta -s-1/2}, and the fact that
\begin{equation*}
\|\xi\|=|\xi|+O(|x|^2|\xi|)
\end{equation*}
imply
\begin{equation*}
[t_{-s-1}]_\alpha{}^\beta(z,\xi)=0\,,
\end{equation*}
where the $t_{-s-k}$, $k=0,1,2,\dots$, are the positively homogeneous
of degree $-s-k$ components of
the full symbol $t$ of the operator $\,\curl\, (-\boldsymbol{\Delta})^{-\frac{s+1}{2}}\,$.
Hence, on account of the expansion
\eqref{prop: expansion of parallel transoport equation 2}
and
formula \eqref{aux4},
integration by parts gives us
\begin{multline}
\label{apt -3 equation 1}
(a^{(s)}_\mathrm{pt})_{-s-3}(z,\xi)=-\frac{i}{6}\, \dfrac{\partial^2\Gamma^\alpha{}_{\sigma\kappa}}{\partial x^\mu \partial x^\nu}(z)\,\dfrac{\partial^3[t_{-s}]_{\alpha}{}^\kappa(z,\xi)}{\partial \xi_\sigma\partial \xi_\mu \partial \xi_\nu}
=
-\frac{i}{6} \,\delta_{\alpha\beta}\,\dfrac{\partial^2\Gamma^\beta{}_{\sigma\kappa}}{\partial x^\mu \partial x^\nu}(z)\,\dfrac{\partial^3[i\varepsilon^{\alpha\gamma\kappa} \xi_\gamma\, |\xi|^{-s-1}]}{\partial \xi_\sigma\partial \xi_\mu \partial \xi_\nu}\,.
\end{multline}
Next, let us notice that in formula \eqref{apt -3 equation 1} one can replace the quantity $\delta_{\alpha\beta}\,\dfrac{\partial^2\Gamma^\beta{}_{\sigma\kappa}}{\partial x^\mu \partial x^\nu}(z)$ with its symmetrised version in the three indices $\sigma$, $\mu$ and $\nu$.
It is well known \cite[Eqn.~(11.9)]{expansions} that the latter is equal to
\begin{equation}
\label{symmetrised christoffel xx}
\frac12 \nabla_\nu \Riem_{\alpha \sigma \mu \kappa}(z)\,.
\end{equation}
Using the symmetries of the Riemann curvature tensor, the expression \eqref{symmetrised christoffel xx} can be equivalently rewritten as
\begin{equation}
\label{symmetrised christoffel yy}
\frac12 \nabla_\nu \Riem_{\alpha \sigma \mu \kappa}
=
\frac14 [\nabla_\nu \Riem_{\alpha \sigma \mu \kappa}+\nabla_\nu \Riem_{\mu \kappa \alpha \sigma}].
\end{equation}
Now, when contracted with
\begin{equation}
\label{apt -3 equation 2}
\dfrac{\partial^3[i\varepsilon^{\alpha\gamma\kappa} \xi_\gamma\, |\xi|^{-s-1}]}{\partial \xi_\sigma\partial \xi_\mu \partial \xi_\nu},
\end{equation}
the quantity \eqref{symmetrised christoffel yy} can be replaced with its symmetrised version in the pair of indices $\mu$ and $\sigma$:
\begin{equation}
\label{apt -3 equation 3}
\frac18 [\nabla_\nu \Riem_{\alpha \sigma \mu \kappa}+\nabla_\nu \Riem_{\mu \kappa \alpha \sigma}
+
\nabla_\nu \Riem_{\alpha \mu \sigma \kappa}+\nabla_\nu \Riem_{\sigma \kappa \alpha \mu}]\,.
\end{equation}
But the quantity \eqref{apt -3 equation 3} is now symmetric in the pair of indices $\alpha$ and $\kappa$, whereas \eqref{apt -3 equation 2} is anti\-symmetric in the same pair of indices. Therefore \eqref{apt -3 equation 1} vanishes.
\end{proof}

%\section{Proof of Proposition~\ref{proposition As is of order -s-3}}
%\label{Proof of Proposition As is of order -s-3}
%
%\textcolor{teal}{To be completed.}

\section{Auxiliary proposition on the leading singularity}
\label{Auxiliary lemma on the leading singularity}

Let us work in Euclidean space $\mathbb{R}^3$ equipped with Cartesian coordinates $y^\alpha$, $\alpha=1,2,3$. Let $\chi:[0,+\infty)\to \mathbb{R}$ be a compactly supported infinitely smooth scalar function such that $\chi=1$ in a neighbourhood of zero. We define a family of functions
$[f^{(s)}]^{\alpha\beta}:\mathbb{R}^3\to\mathbb{R}$ as
\begin{equation}
\label{reference function in Euclidean space}
[f^{(s)}]^{\alpha\beta}(y)
:=
\begin{cases}
0
\quad&\text{if}\quad y=0,
\\
\dfrac{y^\alpha y^\beta-\frac{1}{3}\delta^{\alpha\beta}|y|^2}{|y|^{2-s}}
\,\chi(|y|)
\quad&\text{if}\quad y\ne0,
\end{cases}
\qquad\alpha,\beta=1,2,3.
\end{equation}
Here $s$ is a real parameter.

For $s>0$ the functions $[f^{(s)}]^{\alpha\beta}$ are continuous.
For $s\in(-3,0]$ the functions $[f^{(s)}]^{\alpha\beta}$ are discontinuous at the origin,
but the singularity is integrable.
In what follows we assume that
$
s\in(-3,+\infty)
$
and examine the Fourier transforms
\begin{equation}
\label{reference function in Euclidean space Fourier}
[\hat f^{(s)}]_{\alpha\beta}(\xi)
:=
\int_{\mathbb{R}^3} e^{-iy^\gamma\xi_\gamma}\,[f^{(s)}]_{\alpha\beta}(y)\,\dr y
\end{equation}
of the functions \eqref{reference function in Euclidean space}.
In the RHS of \eqref{reference function in Euclidean space Fourier}
we lowered indices in $f^{(s)}$ using the Euclidean metric
which, in Euclidean space, doesn't change anything.
It is easy to see that the $[\hat f^{(s)}]_{\alpha\beta}:\mathbb{R}^3\to\mathbb{R}$
are bounded infinitely differentiable functions.

Note that the functions $[f^{(s)}]^{\alpha\beta}$ and $[\hat f^{(s)}]_{\alpha\beta}$,
$\alpha, \beta=1,2,3$,
have the following special properties:
\begin{equation*}
\label{special property 1}
[f^{(s)}]^{\alpha\beta}\,\delta_{\alpha\beta}=0\,,
\qquad
[\hat f^{(s)}]_{\alpha\beta}\,\delta^{\alpha\beta}=0\,,
\end{equation*}
\begin{equation*}
\label{special property 2}
\int_{|x|<r}[f^{(s)}]^{\alpha\beta}(y)\,\dr y=0\,,
\qquad
\int_{|\xi|<r}[\hat f^{(s)}]_{\alpha\beta}(\xi)\,\dr\xi=0\,,
\qquad
\forall r>0.
\end{equation*}
%The above identities are straightforward, except perhaps the second identity in \eqref{special property 2}, which follows by observing that
%\begin{multline}
%\label{special property 2 justification}
%\int_{|\xi|<r}[\hat f^{(s)}]_{\alpha\beta}(\xi)\,\dr\xi
%=
%\\
%\int_{\mathbb{R}^3} \frac{8\sin(Rx^1)\sin(Rx^2)\sin(Rx^3)}{x^1 x^2 x^3 |x|^{2-s}} 
%\left[\delta_{\alpha\beta}\left((x^\alpha)^2-\frac13|x|^2\right)+ (1-\delta_{\alpha\beta}) x^\alpha x^\beta \right]
%\,\chi(|x|)\,\dr x\,.
%\end{multline}
%Here we do not sum over $\alpha$ and $\beta$. Both integrals in the RHS of \eqref{special property 2 justification} vanish by symmetry arguments.

The following proposition is the main result of this appendix.

\begin{proposition}
\label{reference function lemma}
We have
\begin{equation}
\label{reference function lemma equation 1}
[\hat f^{(s)}]_{\alpha\beta}(\xi)
=
c(s)\,
\frac{\xi_\alpha\xi_\beta-\frac{1}{3}\delta_{\alpha\beta}|\xi|^2}{|\xi|^{5+s}}
+
O(|\xi|^{-\infty})\quad\text{as}\quad|\xi|\to+\infty\,,
\end{equation}
where $c(s)$ is defined in accordance with~\eqref{The reference operator equation 5} and the remainder, together with its partial derivatives in $\xi$ of any order, is uniform in $s$ over any closed bounded interval
\begin{equation}
\label{reference function lemma equation 2}
[a,b]\subset(-3,+\infty).
\end{equation}
\end{proposition}

The proof of Proposition~\ref{reference function lemma} relies on the following two auxiliary lemmata.
The $\xi$ in these two lemmata is 1-dimensional, i.e.~a real number.
Furthermore, we assume that $\xi>0$.

\begin{lemma}
\label{lemma dominican 1}
Let $t>-2$. Then
\begin{equation}
\label{Dominican 0}
\int_0^{+\infty}y^t\,\sin(y\xi)\,\chi(y)\,\dr y
=
\begin{cases}
\frac{1}{\xi^{t+1}}\Gamma(t+1) \cos\left( \frac{\pi t}{2}\right) &\text{if }t\neq -1,
\\
\frac\pi2 &\text{if }t=-1,
\end{cases}
+O(\xi^{-\infty})
\quad\text{as}\quad\xi\to+\infty\,,
\end{equation}
where the remainder, together with its derivatives in $\xi$ of any order, is uniform in $t$ over any closed bounded interval
$[a,b]\subset(-2,+\infty)$.
\end{lemma}
\begin{proof}
We have
\begin{multline}
\label{Dominican 1}
\int_0^{+\infty}y^t\,\sin(y\xi)\,\chi(y)\,\dr y
=
\lim_{\epsilon\to0^+}
\int_0^{+\infty}y^t\,\sin(y\xi)\,\chi(y)\,e^{-\epsilon y}\,\dr y
\\
=
\lim_{\epsilon\to0^+}
\int_0^{+\infty}y^t\,\sin(y\xi)\,e^{-\epsilon y}\,\dr y
+
\lim_{\epsilon\to0^+}
\int_0^{+\infty}y^t\,\sin(y\xi)\,(\chi(y)-1)\,e^{-\epsilon y}\,\dr y\,.
\end{multline}
Straightforward integration by parts gives us
\begin{equation}
\label{Dominican 2}
\lim_{\epsilon\to0^+}
\int_0^{+\infty}y^t\,\sin(y\xi)\,(\chi(y)-1)\,e^{-\epsilon y}\,\dr y
=O(\xi^{-\infty})
\quad\text{as}\quad\xi\to+\infty\,,
\end{equation}
where the remainder, together with its derivatives in $\xi$ of any order, is uniform in $t$ over any closed bounded interval
$[a,b]\subset(-2,+\infty)$.
We also have
\begin{equation}
\label{Dominican 3 alternative}
\int_0^{+\infty}y^t\,\sin(y\xi)\,e^{-\epsilon y}\,\dr y
=
\begin{cases}
\left(\xi ^2+\epsilon^2\right)^{-\frac{t+1}{2}}
\Gamma (t+1)
\sin
\left(
(t+1)
\arctan\left(\frac{\xi }{\epsilon}\right)
\right)
&\text{if }t\neq-1,
\\
\arctan\left(\frac{\xi }{\epsilon}\right)
&\text{if }t=-1.
\end{cases}
\end{equation}
Substituting \eqref{Dominican 2} and \eqref{Dominican 3 alternative} into \eqref{Dominican 1},
we arrive at \eqref{Dominican 0}.
\end{proof}

\begin{lemma}
\label{lemma dominican 2}
Let $t>-4$. Then
\begin{multline*}
\label{Dominican 4}
\int_0^{+\infty}y^t\,\left(\sin(y\xi)-y \xi\cos(y\xi)\right)\,\chi(y)\,\dr y
\\
=
\begin{cases}
\frac{t+2}{\xi^{t+1}}\Gamma(t+1) \cos\left( \frac{\pi t}{2}\right) &\text{if }t\neq -1,-2,-3,
\\
\frac\pi2 &\text{if }t=-1,
\\
\xi &\text{if }t=-2,
\\
\frac\pi4 \xi^2 &\text{if }t=-3,
\end{cases}
+O(\xi^{-\infty})
\quad\text{as}\quad\xi\to+\infty\,,
\end{multline*}
where the remainder, together with its derivatives in $\xi$ of any order, is uniform in $t$ over any closed bounded interval
$[a,b]\subset(-4,+\infty)$.
\end{lemma}

\begin{proof}
The proof is similar to that of Lemma~\ref{lemma dominican 1}, hence omitted.
\end{proof}

\

\begin{proof}[Proof of Proposition~\ref{reference function lemma}]
The quantity \eqref{reference function in Euclidean space Fourier} is covariant under Euclidean rotations and reflections,
so it suffices to prove \eqref{reference function lemma equation 1} in the special case
\begin{equation}
\label{30 March 2024 equation 1}
\xi_\alpha=\lambda\,\delta_\alpha{}^3\,,
\qquad\lambda\to+\infty\,.
\end{equation}

When $\alpha\ne\beta$ both the LHS and RHS of \eqref{reference function lemma equation 1}
vanish, so it suffices to deal with the special cases
\begin{equation}
\label{30 March 2024 equation 2}
\alpha=\beta=1,
\end{equation}
\begin{equation}
\label{30 March 2024 equation 3}
\alpha=\beta=2,
\end{equation}
\begin{equation}
\label{30 March 2024 equation 4}
\alpha=\beta=3.
\end{equation}
Case \eqref{30 March 2024 equation 2} reduces to case \eqref{30 March 2024 equation 3}
by means of Euclidean rotations and reflections, so further on we deal with
\eqref{30 March 2024 equation 3} and \eqref{30 March 2024 equation 4}.

Assuming \eqref{30 March 2024 equation 1} and \eqref{30 March 2024 equation 3}, and performing the change of variables $$(y_1,y_2,y_3)\mapsto (\sqrt{r^2-y^2-z^2},y,z)$$ in \eqref{reference function in Euclidean space Fourier}, we obtain
\begin{multline}
\label{proof reference function lemma equation 2}
[\hat f^{(s)}]_{22}(\lambda)
=
\int_0^{+\infty} \left(2\int_{-r}^r \cos(\lambda z) \left(2 \underset{=\frac\pi{12}(r^2-3z^2)}{\underbrace{\int_0^{\sqrt{r^2-z^2}}
\frac{y^2-\frac13 r^2}{\sqrt{r^2-z^2-y^2}}dy}}\right)dz\right) r^{s-1}\,\chi(r)\,\dr r
\\
=
\frac\pi3\int_0^{+\infty} \left(\underset{=-4 \left[3r\lambda^{-2}\cos(\lambda r)+(r^2\lambda^{-1}-3\lambda^{-3})\sin(\lambda r) \right]}{\underbrace{\int_{-r}^r (r^2-3z^2)\cos(\lambda z)\,dz}}\right) r^{s-1}\,\chi(r)\,\dr r
\\
=
-\frac{4\pi}{3}\int_0^{+\infty} \left(r^2\lambda^{-1}\sin(\lambda r)+3r\lambda^{-2}\cos(\lambda r)-3\lambda^{-3}\sin(\lambda r)\right) r^{s-1}\,\chi(r)\,\dr r\,.
\end{multline}

A similar argument for the case \eqref{30 March 2024 equation 1} and \eqref{30 March 2024 equation 4} gives us
\begin{equation}
\label{proof reference function lemma equation 3}
[\hat f^{(s)}]_{33}(\lambda)
=
\frac{8\pi}{3}\int_0^{+\infty} \left(r^2\lambda^{-1}\sin(\lambda r)+3r\lambda^{-2}\cos(\lambda r)-3\lambda^{-3}\sin(\lambda r)\right) r^{s-1}\,\chi(r)\,\dr r\,.
\end{equation}
By combining formulae~\eqref{proof reference function lemma equation 2} and~\eqref{proof reference function lemma equation 3} with Lemmata~\ref{lemma dominican 1} and~\ref{lemma dominican 2} one arrives at~\eqref{reference function lemma equation 1}.

Finally, let us explain why
the remainder in formula \eqref{reference function lemma equation 1},
together with its partial derivatives in $\xi$ of any order, is uniform in $s$.
The argument presented above establishes only uniformity under differentiations
along $|\xi|$, i.e.~in the radial direction, so we need to deal with mixed derivatives.
Uniformity under mixed differentiations follows from the observation that
the remainder term in formula \eqref{reference function lemma equation 1}
is covariant under Euclidean rotations.
Indeed, the remainder term in formula \eqref{reference function lemma equation 1}
can be written as a composition of a function of $|\xi|$ with an orthogonal matrix
describing the rotation
(this matrix appears twice in the composition,
acting separately on the tensor indices $\alpha$ and $\beta$).
And the orthogonal matrix describing the rotation can be expressed, say,
in terms of Euler angles or Cardan angles (pitch, yaw, and roll).
In the end, any mixed derivative
of the remainder in formula \eqref{reference function lemma equation 1}
can be written as a linear combination of radial derivatives.
\end{proof}

\begin{remark}
\label{reference function lemma remark}
One might think that a simpler way of proving Lemma~\ref{reference function lemma}
would be to consider the function $\,|y|^{2+s}\,\chi(|y|)\,$,
compute its Fourier transform and then recover \eqref{reference function lemma equation 1}
by performing two partial differentiations in $y$.
However, such an approach does not to lead to uniformity in~$s$ of the remainder as $s\to0$.
The problem here is that the Fourier transform of the function $\,|y|^{2}\,\chi(|y|)\,$
has asymptotic expansion $O(|\xi|^{-\infty})$, namely, it is superpolynomially decaying.
Of course, superpolynomial decay at $s=0$ can be avoided by switching from
$\,|y|^{2}\,\chi(|y|)\,$ to $\,|y|^{2}\,\ln|y|\,\chi(|y|)\,$, but this would mean
working with a function given by two different formulae depending on the value of
the parameter $s$:
$\,|y|^{2+s}\,\chi(|y|)\,$ for $s\ne0$
and
$\,|y|^{2}\,\ln|y|\,\chi(|y|)\,$ for $s=0$.
This clearly destroys uniformity of the remainder as $s\to0$.
\end{remark}

\end{appendices}


\begin{thebibliography}{42}
\addcontentsline{toc}{section}{References}

%\bibitem{AS}
%M.~Abramowitz and I.A.~Stegun, \emph{Handbook of Mathematical Functions: With Formulas, Graphs, and Mathematical Tables}. Dover Books on Advanced Mathematics, Dover Publications, 1965.
%
\bibitem{asymm1}
M.F.~Atiyah, V.K.~Patodi and I.M.~Singer, 
Spectral asymmetry and Riemannian geometry,
{\it Bull. London Math. Soc.} \textbf{5} (1973) 229--234.
DOI: \href{https://doi.org/10.1112/blms/5.2.229}{10.1112/blms/5.2.229}.

\bibitem{asymm2}
M.F.~Atiyah, V.K.~Patodi and I.M.~Singer, 
Spectral asymmetry and Riemannian geometry I,
{\it Math. Proc. Camb. Phil. Soc.} \textbf{77} (1975) 43--69.
DOI: \href{https://doi.org/10.1017/S0305004100049410}{10.1017/S0305004100049410}.

\bibitem{asymm3}
M.F.~Atiyah, V.K.~Patodi and I.M.~Singer, 
Spectral asymmetry and Riemannian geometry II,
{\it Math. Proc. Camb. Phil. Soc.} \textbf{78} (1975) 405--432.
DOI: \href{https://doi.org/10.1017/S0305004100051872}{10.1017/S0305004100051872}.

\bibitem{asymm4}
M.F.~Atiyah, V.K.~Patodi and I.M.~Singer, 
Spectral asymmetry and Riemannian geometry III,
{\it Math. Proc. Camb. Phil. Soc.} \textbf{79} (1976) 71--99.
DOI: \href{https://doi.org/10.1017/S0305004100052105}{10.1017/S0305004100052105}.


%\bibitem{avakumovic}
%V.~G.~Avakumovic, 
%{\"U}ber die Eigenfunktionen auf geschlossenen Riemannschen Mannigfaltigkeiten, 
%{\em Math. Z.} \textbf{65} (1956) 327--344.
%
%\bibitem{AFV}
%Z.~Avetisyan, Y.-L.~Fang and D.~Vassiliev,
%Spectral asymptotics for first order systems,
%{\it J.~Spectr.~Theory} \textbf{6} no.~4 (2016) 695--715. 

%
%\bibitem{spin1}
%Z.~Avetisyan, Y.-L.~Fang, N.~Saveliev and D.~Vassiliev,
%Analytic definition of spin structure.
%{\it J.~Math.~Phys.} \textbf{58} (2017), 082301.
%
%\bibitem{ASV}
%Z.~Avetisyan, J.~Sj\"ostrand and D.~Vassiliev,
%The second Weyl coefficient for a first order system,
%in: {\it Analysis as a tool in mathematical physics}, P.~Kurasov, A.~Laptev, S.~Naboko and B.~Simon (Eds.), {\it Operator Theory: Advances and Applications} \textbf{276} Birkh\"auser Verlag (2020) 120--153.
%
%
%\bibitem{baer92}
%C.~B\"ar, 
%The Dirac operator on homogeneous spaces and its spectrum on 3-dimensional lens spaces,
%{\it Arch. Math.} {\bf 59} (1992) 65--79 (1992).
%

%%\bibitem{Bar3}
%%C.~B\"ar, 
%%Metrics with harmonic spinors,
%%{\it Geom.~Funct.~Anal.} \textbf{6} (1996) 899--942.
%%
%%
%%\bibitem{Bar1}
%%C.~B\"ar, 
%%The Dirac operator on space forms of positive curvature, {\it J.~Math.~Soc.~Japan} \textbf{48} (1996) 69--83.
%%
%%\bibitem{Bar2}
%%C.~B\"ar, 
%%Dependence of the Dirac spectrum on the spin structure, 
%%{\it S{\'e}min. Congr.} \textbf{4} (2000) 17--33.
%%

\bibitem{baer_curl}
C.~B\"ar, 
The curl operator on odd-dimensional manifolds,
{\it J. Math. Phys.} {\bf 60} (2019) 031501.
DOI: \href{https://doi.org/10.1063/1.5082528}{10.1063/1.5082528}.

%%\bibitem{battistotti}
%%P.~Battistotti, 
%%{\it An invariant approach to symbolic calculus for pseudodifferential operators on manifolds}, 
%%PhD thesis, King's College London (2015).
%
%%\bibitem{baum}
%%P.~Baum and R.~G.~Douglas,
%%Toeplitz operators and Poincare duality, 
%%in: {\it Toeplitz Centennial}, I.~Gohberg (Eds.), {\it Operator Theory: Advances and Applications} {\bf 4} Birkh\"auser, Basel (1981) 137--166.
%%
%


%\bibitem{birman1}
%M.~Sh.~Birman and M.~Z.~Solomyak,
%Asymptotic behavior of the spectrum of pseudodifferential operators with anisotropically homogeneous symbols, (Russian) {\it Vestnik Leningrad. Univ. Mat. Mekh. Astronom.} {\bf 13} no.~3 (1977) 13--21. English translation: {\it Vestn. Leningr. Univ. Math.} {\bf 10} (1982) 237--247.
%
%\bibitem{birman2}
%M.~Sh.~Birman and M.~Z.~Solomyak,
%Asymptotic behavior of the spectrum of pseudodifferential operators with anisotropically homogeneous symbols II, (Russian) {\it Vestnik Leningrad. Univ. Mat. Mekh. Astronom.} {\bf 13} no.~3 (1979) 5--10. English translation: {\it Vestn. Leningr. Univ. Math.} {\bf 12} (1980) 155--161.
%
%%\bibitem{birman}
%%M.~Sh.~Birman and M.~Z.~Solomyak,
%%On subspaces that admit a pseudodifferential projector, (Russian) {\it Vestnik Leningrad. Univ. Mat. Mekh. Astronom.} \textbf{82} no.~1 (1982) 18--25. English translation: {\it Vestnik Leningrad. Univ. Mat. Mekh. Astronom.} \textbf{15} (1982) 17--27. 

%\bibitem{birman_curl}
%M.~Sh.~Birman and M.~Z.~Solomyak,
%The Weyl asymptotics of the spectrum of the Maxwell operator for domains with a Lipschitz boundary, 
%{\it Vestn. Leningr. Univ. Math.} {\bf 20} no.~3 (1987) 15--21.
%
%\bibitem{birman_curl2}
%M.~Sh.~Birman and M.~Z.~Solomyak,
%$L^2$-Theory of the Maxwell operator in arbitrary domains, 
%{\it Uspekhi Mat. Nauk} {\bf 42} no.~6 (1987) 61--76; {\it Russian Math. Surveys} {\bf 42} no.~6 (1987) 75--96.

\bibitem{bismut}
J.-M.~Bismut and D.~S.~Freed, 
The analysis of elliptic families. II. Dirac operators, eta invariants, and the holonomy theorem, 
{\it Comm. Math. Phys.} \textbf{107} (1986) 103--163.
DOI: \href{https://doi.org/10.1007/BF01206955}{10.1007/BF01206955}.

%\bibitem{bolte}
%J.~Bolte and R.~Glaser, 
%Semiclassical Egorov theorem and quantum ergodicity for matrix valued operators,
%{\it Comm.~Math.~ Phys. } {\bf 247} (2004) 391--419.
%%
%%\bibitem{woj3}
%%B.~Booss-Bavnbek and K.~P.~Wojciechowski,
%%Desuspension of splitting elliptic symbols. I. 
%%{\it Ann.~Global Anal.~Geom.} \textbf{3} no.~3 (1985) 337--383. 
%%
%%\bibitem{woj4}
%%B.~Booss-Bavnbek and K.~P.~Wojciechowski,
%%Desuspension of splitting elliptic symbols. II. 
%%{\it Ann.~Global Anal.~Geom.} \textbf{4} no.~3 (1986) 349--400. 
%%
%%\bibitem{woj5}
%%B.~Booss-Bavnbek and K.~P.~Wojciechowski,
%%Pseudo-differential projections and the topology of certain spaces of elliptic boundary value problems,
%%{\it Comm.~Math.~Phys.} \textbf{121} no.~1 (1989) 1--9. 
%%
%%\bibitem{bourguignon}
%%J.-P.~Bourguignon and P.~Gauduchon, 
%%Spineurs, op{\'e}rateurs de dirac et variations de m{\'e}triques, {\it Comm.~Math.~Phys.} \textbf{144} (1992) 581--599.

\bibitem{5terms}
G.~Bracchi, M.~Capoferri and D.~Vassiliev,
Higher order Weyl coefficients for the operator curl.
\emph{In preparation}.

%\bibitem{branson}
%T.~P.~Branson and P.~B.~Gilkey, Residues of the eta function for an operator
%of Dirac type,
%{\it J.~Funct.~Anal.} \textbf{108} (1992) 47--87.


\bibitem{expansions}
L.~Brewin,
Riemann normal coordinate expansions using Cadabra.
Preprint \href{https://doi.org/10.48550/arXiv.0903.2087}{arXiv:0903.2087v3}.

%\bibitem{brummelhuis}
% R. ~Brummelhuis and J.~Nourrigat,
% Scattering amplitude for dirac operators, 
% {\it Comm.  Partial Differential Equations} {\bf 24} no.~1-2 (1999) 377--394.
% 
% %
%\bibitem{BR99}
%V.~Bruneau and D.~Robert,
%Asymptotics of the scattering phase for the Dirac operator: High energy, semi-classical and non-relativistic limits,
%{\it Ark.  Mat. } {\bf 37} (1999) 1--32.
%

\bibitem{bruning}
J.~Br\"uning and M.~Lesch,
On the $\eta$-invariant of certain nonlocal boundary value problems,
{\it Duke Math. J.} {\bf 96} no.~2 (1999) 425--468. 
DOI: \href{https://doi.org/10.1215/S0012-7094-99-09613-8}{10.1215/S0012-7094-99-09613-8}.

\bibitem{diagonalisation}
M.~Capoferri,
Diagonalization of elliptic systems via pseudodifferential projections,
{\it J. Differential Equations} {\bf 313} (2022) 157--187.
DOI: \href{https://doi.org/10.1016/j.jde.2021.12.032}{10.1016/j.jde.2021.12.032}.

\bibitem{dirac_asymmetry}
M.~Capoferri, B.~Costeri, C.~Dappiaggi,
Spectral asymmetry via pseudodifferential projections: the massless Dirac operator.
Preprint \href{https://doi.org/10.48550/arXiv.2504.02336}{arXiv:2504.02336} (2025).

%%
%\bibitem{lorentzian}
%M.~Capoferri, C.~Dappiaggi and N.~Drago,
%Global wave parametrices on globally hyperbolic spacetimes,
%{\it J.~Math.~Anal.~Appl.} {\bf 490} (2020) 124316.
%
%\bibitem{wave}
%M.~Capoferri, M.~Levitin and D.~Vassiliev,
%Geometric wave propagator on Riemannian manifolds,
%\emph{Comm. Anal. Geom.} {\bf 30} no.~8 (2022) 1713--1777.
%
%\bibitem{obstructions}
%M.~Capoferri, G.~Rozenblum, N.~Saveliev and D.~Vassiliev,
%Topological obstructions to the diagonalisation of pseudodifferential systems
%{\it Proc. Amer. Math. Soc. (Ser. B)} {\bf 9} (2022) 472--486.

%\bibitem{sesqui}
%M.~Capoferri, N.~Saveliev and D.~Vassiliev,
%Classification of first order sesquilinear forms,
%{\it Rev.~Math.~Phys.} \textbf{32} (2020) 2050027.

%\bibitem{diffeo}
%M.~Capoferri and D.~Vassiliev,
%Spacetime diffeomorphisms as matter fields,
%{\it J. Math. Phys.} {\bf 61} (2020) 111508.

%\bibitem{dirac}
%M.~Capoferri and D.~Vassiliev,
%Global propagator for the massless Dirac operator and spectral asymptotics.
%{\it Integral Equations Operator Theory} {\bf 94} (2022) 30.

\bibitem{part1}
M.~Capoferri and D.~Vassiliev,
Invariant subspaces of elliptic systems I: pseudodifferential projections.
{\it J.~Funct.~Anal.} {\bf 282} no.~8 (2022) 109402.
DOI: \href{https://doi.org/10.1016/j.jfa.2022.109402}{10.1016/j.jfa.2022.109402}.

\bibitem{part2}
M.~Capoferri and D.~Vassiliev,
Invariant subspaces of elliptic systems II: spectral theory,
{\it J.  Spectr. Theory} {\bf 12} no.~1 (2022) 301--338.
DOI: \href{https://doi.org/10.4171/JST/402}{10.4171/JST/402}.


\bibitem{curl}
M.~Capoferri and D.~Vassiliev,
Beyond the Hodge theorem: curl and asymmetric pseudodifferential projections,
{\it J. London Math. Soc.}, to appear.
Preprint \href{https://doi.org/10.48550/arXiv.2309.02015}{arXiv:2309.02015} (2023).




%\bibitem{CDV}
%O.~Chervova, R.~J.~Downes and D.~Vassiliev,
%The spectral function of a first order elliptic system,
%{\it J. Spectr. Theory} \textbf{3} no.~3 (2013) 317--360. 

%%\bibitem{jst_part_b}
%%O.~Chervova, R.~J.~Downes and D.~Vassiliev,
%%Spectral theoretic characterization of the massless Dirac operator.
%%{\it J.~London Math.~Soc.} \textbf{89} (2014) 301--320.
%%
%
%\bibitem{cordes3}
%H.~O.~Cordes, 
%A version of Egorov’s theorem for systems of hyperbolic pseudo-differential equations,
%{\it J. Funct. Anal.} {\bf 48} no.~3 (1982) 285--300.
%
%%
%\bibitem{cordes}
%H.~O.~Cordes, 
%A pseudodifferential-Foldy-Wouthuysen transform,
%{\it Comm.  Partial Differential Equations} {\bf 8} (1983) 1475--1485.
%

%\bibitem{filonov_curl1}
%M.N.~Demchenko and N.~Filonov, 
%Spectral asymptotics of the Maxwell operator on Lipschitz manifolds with boundary, in: \emph{Spectral Theory of Differential Operators: M.Sh.~Birman 80th Anniversary Collection} (Amer. Math. Soc., Providence, RI, 2008), 73--90.

%\bibitem{DLMF}
%\emph{NIST Digital Library of Mathematical Functions}. \url{https://dlmf.nist.gov/}, Release 1.2.0 of 2024-03-15. F.W.J.~Olver, A.B.~Olde Daalhuis, D.W.~Lozier, B.I.~Schneider, R.F.~Boisvert, C.W.~Clark, B.R.~Miller, B.V.~Saunders, H.S.~Cohl, and M.A.~McClain, eds.

\bibitem{Dowker}
J.S.~Dowker, private communication 19 January 2025.

%%\bibitem{torus}
%%R.~J.~Downes, M.~Levitin and D.~Vassiliev,
%%Spectral asymmetry of the massless Dirac operator on a 3-torus, 
%%{\it J.~Math.~Phys.} \textbf{54} (2013).
%
%\bibitem{DuGu}
%J.~J.~Duistermaat and V.~W.~Guillemin,
%The spectrum of positive elliptic operators and periodic bicharacteristics,
%{\it Invent. Math.} \textbf{29}  no.~1 (1975) 39--79.
%%
%\bibitem{DuHo}
%J.~J.~Duistermaat and L.~H\"ormander,
%Fourier integral operators. II.,
%{\it Acta Math.} \textbf{128} no.~3--4 (1972) 183--269.
%
%\bibitem{peralta1}
%A.~Enciso, W.~Gerner and D.~Peralta-Salas,
%Optimal convex domains for the first curl eigenvalue,
%{\emph Trans. Amer. Math. Soc.}, in press.
%
%\bibitem{peralta2}
%A.~Enciso and D.~Peralta-Salas,
%Non-existence of axisymmetric optimal domains with smooth boundary for the first curl eigenvalue, {\it Ann. Sc. Norm. Super. Pisa Cl. Sci.} {\bf XXIV} (2023) 311-327.
%
%\bibitem{sphere}
%Y.-L.~Fang, M.~Levitin and D.~Vassiliev,
%Spectral analysis of the Dirac operator on a 3-sphere,
%{\it Operators and Matrices} \textbf{12} (2018) 501--527.
%
%\bibitem{fang_PhD}
%Y.-L. Fang,
%{\it Analysis of first order systems on manifolds without boundary: A spectral theoretic approach}.
%PhD Thesis, University College London, 2017.
%URL: \url{https://discovery.ucl.ac.uk/id/eprint/1560979}.

%%\bibitem{israel}
%%Y.-L.~Fang and D.~Vassiliev,
%%Analysis of first order systems of partial differential equations.
%%In
%%{\it Complex Analysis and Dynamical Systems VI: Part~1:
%%PDE, Differential Geometry, Radon Transform}.
%%AMS Contemporary Mathematics series
%%\textbf{653} (2015), 163--176.
%%

%\bibitem{filonov_curl2}
%N.~Filonov, 
%Weyl asymptotics of the spectrum of the Maxwell operator in Lipschitz domains of arbitrary dimension, 
%{\it Algebra i Anal.} {\bf 25} no.~1 (2013) 170--215.

%%\bibitem{friedrich}
%%T.~Friedrich,
%%{\it Dirac Operators in Riemannian Geometry},
%%Graduate Studies in Mathematics \textbf{25}, American Mathematical Society (2000).
%%

\bibitem{gilkey_local}
P.~B.~Gilkey, 
The residue of the local $\eta$ function at the origin,
{\it Math. Ann.} \textbf{240} (1973) 183--190.
DOI: \href{https://doi.org/10.1007/BF01364633}{10.1007/BF01364633}.

%%%
%%\bibitem{gilkey}
%%P.~B.~Gilkey, 
%%The residue of the global $\eta$ function at the origin,
%%{\it Adv. Math.} \textbf{40} (1981) 290--307.

%% 
%%\bibitem{gohberg}
%%I.~Gohberg and M.~G.~Krein, 
%%Systems of integral equations on the half-line with kernels depending on the difference of the arguments, 
%%{\it Uspekhi Matem. Nauk} {\bf 13} no.~2 (1958) 3--72.
%

%\bibitem{gureev}
%T.E.~Gureev, 
%Exact asymptotics of the spectrum of the Maxwell operator in a solid resonator, {\it Funct. Anal. Appl.} {\bf 24} no.~3 (1990) 235--237.

%%
%%\bibitem{heath-brown}
%%D.~R.~Heath-Brown, 
%%Lattice Points in the Sphere,
%%{\it Number Theory in Progress} Vol.~2, de Gruyter, Berlin (1999) 883--892.
%%
%%
%\bibitem{hintz}
%P.~Hintz, 
%Resonance expansions for tensor-valued waves on asymptotically Kerr--de Sitter spaces,
%{\it J. Spectr. Theory} {\bf 7} no.~2 (2017) 519--557.
%
%\bibitem{hitchin}
%N.~Hitchin,
%Harmonic spinors,
%{\it Adv. Math.} {\bf 14} no.~1 (1974) 1--55.
%
%\bibitem{Hor 1968}
%L.~H\"ormander,
%The spectral function of an elliptic operator,
%{\it Acta Math.} \textbf{121} (1968) 193--218.

%%\bibitem{Hor}
%%L.~H\"ormander,
%%{\it The analysis of linear partial differential operators. I}.
%%Reprint of the second (1990) edition. Classics in Mathematics. Springer-Verlag, Berlin, 2003;
%%{\it III}. Reprint of the 1994 edition. Classics in Mathematics. Springer-Verlag, Berlin, 2007; {\it IV}. Reprint of the 1994 edition. Classics in Mathematics. Springer-Verlag, Berlin, 2009.
%
%\bibitem{Horn}
%R.~A.~Horn and C.~R.~Johnson,
%{\it Matrix Analysis} (2nd Edition),
%Cambridge University Press, 2012.
%
%\bibitem{Ivr80}
%V.~Ivrii,
%Second term of the spectral asymptotic expansion of the Laplace--Beltrami operator on manifolds with boundary,
%{\it Funct. Anal. Appl.} \textbf{14} (1980) 98--106.
%

%\bibitem{Ivr82}
%V.~Ivrii,
%Accurate spectral asymptotics for elliptic operators that act in vector bundles
%{\it Funct. Anal. Appl.} \textbf{16} (1982) 101--108.

%
%\bibitem{Ivr84}
%V.~Ivrii,
%{\it Precise spectral asymptotics for elliptic operators
%acting in fiberings over manifolds with boundary},
%Lecture Notes in Mathematics \textbf{1100}, Springer-Verlag, Berlin, 1984.
%
%\bibitem{Ivr98}
%V.~Ivrii,
%{\it Microlocal analysis and precise spectral asymptotics},
%Springer-Verlag, Berlin, 1998.
%%
%%
%%\bibitem{Ivr19}
%%V.~Ivrii,
%%{\it Microlocal analysis, sharp spectral asymptotics and applications II},
%%Springer International Publishing, 2019.
%

%\bibitem{JS}
%D.~Jakobson and A.~Strohmaier,
%High energy limits of Laplace-type and Dirac-type eigenfunctions and frame flows,
%{\it Comm. Math. Phys.} {\bf 270} (2007) 813--833.

\bibitem{jost}
J.~Jost, 
{\it Riemannian Geometry and Geometric Analysis},
Springer-Verlag,  2011.
DOI: \href{https://doi.org/10.1007/978-3-319-61860-9}{10.1007/978-3-319-61860-9}.

%\bibitem{milnor2}
%M.A.~Kervaire and J.W.~Milnor,
%Groups of homotopy spheres: I,
%\emph{Ann. of Math.} \textbf{77} no.~3 (1963) 504--537.

%%\bibitem{Kirby}
%%R.~C.~Kirby,
%%{\it The topology of 4-manifolds}.
%%Lecture Notes in Mathematics, 1374, Springer-Verlag, Berlin, 1989.
%%
%\bibitem{LSV}
%A.~Laptev, Yu.~Safarov and D.~Vassiliev,
%On global representation of Lagrangian distributions and solutions of hyperbolic equations,
%{\it Comm.~Pure~Appl. Math.} \textbf{47}  no.~11 (1994) 1411--1456.

%\bibitem{lauret}
%E.A.~Lauret, 
%The smallest Laplace eigenvalue of homogeneous 3-spheres,
%{\it Bull.~LMS} {\bf 51} no.~1 (2019) 49--69.

%%
%%\bibitem{Lawson}
%%H.~B.~Lawson and M.-L.~Michelsohn, 
%%{\it Spin Geometry}, 
%%Princeton University Press, Princeton (1989).
%%

%\bibitem{lax}
%P.D.~Lax,
%{\it Functional Analysis},
%Wiley, 2002.

%\bibitem{lerner}
%N.~Lerner and F.~Vigneron,
%On some properties of the curl operator and their consequences for the Navier-Stokes system. ArXiv preprint arXiv:2203.07950 (2022).
%
%\bibitem{lipkin}
%D.M.~Lipkin,
%Existence of a new conservation law in electromagnetic theory,
%{\it J. Math. Phys.} {\bf 5} (1964) 696--700.


%\bibitem{levitan}
%B.~M.~Levitan,
%On the asymptotic behaviour of the spectral function of a self-adjoint differential second order equation,
%{\em Izv. Akad. Nauk SSSR Ser. Mat.} \textbf{19} (1952) 325--352.
%%
%%\bibitem{strohmaier}
%%L.~Li and A.~Strohmaier,
%%The local counting function of operators of Dirac and Laplace type,
%%{\it J.~Geom.~Phys.} \textbf{104} (2016) 204--228.
%
%\bibitem{LF91}
%R.~G.~Littlejohn and W.~G.~Flynn,  
%Geometric phases in the asymptotic theory of coupled wave equations,
%{\it Phys.  Rev.  A} {\bf 44} (1991) 5239--5256.
%
%

%\bibitem{lotay}
%J.~Lotay,
%Stability of coassociative conical singularities,
%\emph{Comm. Anal. Geom.} {\bf 20} no.~4 (2012) 803--867.
%

\bibitem{loya}
P.~Loya,
Residue Traces and Kernel Expansions of Pseudodifferential Operators, 
{\it Modern trends in geometry and topology},  Cluj Univ.\ Press, Cluj-Napoca (2006) 251--264.

%%%%\bibitem{kratzel}
%%%E.~Kr\"atzel,
%%%{\it Analytische Funktionen in der Zahlentheorie}.
%%%Teubner-Texte zur Mathematik 139, Vieweg+Teubner Verlag, 2000.
%%

%\bibitem{millson}
%J.J.~Millson, 
%\emph{Chern-Simons invariants of constant curvature manifolds}. Ph.D. thesis, University of California, Berkeley, 1973.
%
%\bibitem{milnor1}
%J.W.~Milnor,
%On manifolds homeomorphic to the 7-sphere,
%\emph{Ann. of Math.} \textbf{64} no.~2 (1956) 399--405.
%
%\bibitem{minak}
%S.~Minakshisundaram and {{\AA}}.~Pleijel,
%Some properties of the eigenfunctions of the Laplace-operator on Riemannian manifolds,
%{\it Canadian J. Math.} \textbf{1} (1949) 242--256.
%
%\bibitem{Uwe Muller}
%U.~Muller, C.~Schubert and A.E.M.~van de Ven,
%A closed formula for the Riemann normal coordinate expansion,
%{\it Gen. Relativ. Gravit.} \textbf{31} (1999) 1759--1768.

%%\bibitem{nakahara}
%%M.~Nakahara,
%%{\it Geometry, Topology and Physics}, 2nd Edition, 
%%IOP Publishing (2003).
%
%\bibitem{NS04}
%G.~Nenciu and V.~Sordoni,
%Semiclassical limit for multistate Klein–Gordon systems: almost invariant subspaces and scattering theory,
%{\it J. Math.  Phys.}  {\bf 45} (2004) 3676.
%
%\bibitem{nicoll}
%W.~J.~Nicoll, 
%{\it Global oscillatory integrals for solutions of hyperbolic systems},
%PhD thesis, University of Sussex (1998).
%
%\bibitem{PST03}
%G.~Panati,  H.~Spohn and S.~Teufel, 
%Space-adiabatic perturbation theory,
%{\it Adv.  Theor.  Math.  Phys. } {\bf 7} (2003) 145--204.
%
%\bibitem{rellich}
%F.~Rellich,
%{\it Perturbation theory of eigenvalue problems},
%Courant Institute of Mathematical Sciences, New York University, 1954.
%
%\bibitem{safarov_curl}
%Yu.~Safarov,
%Asymptotic behavior of the spectrum of the Maxwell operator, 
%{\it J. Sov. Math.} {\bf 27} (1984) 2655--2661.

%\bibitem{safarov}
%Yu.~Safarov,
%{\it Non-classical two-term spectral asymptotics for self-adjoint elliptic operators}.
% DSc~thesis, Leningrad Branch of the Steklov Mathematical Institute of the USSR Academy of Sciences (1989). In Russian.
%
%\bibitem{SaVa}
%Yu.~Safarov and D.~Vassiliev,
%{\it The asymptotic distribution of eigenvalues of partial differential operators},
%Amer.~Math.~Soc., Providence (RI), 1997.
%
%%%\bibitem{sasaki}
%%%S.~Sasaki, 
%%%On the differential geometry of tangent bundles of Riemannian manifolds, 
%%%{\it Tohoku Math.~J.}~\textbf{10} no.~3 (1958) 338--354.
%%%
%%%\bibitem{sasaki2}
%%%S.~Sasaki, 
%%%On the differential geometry of tangent bundles of Riemannian manifolds~II, 
%%%{\it Tohoku Math.~J.}~\textbf{14} no.~2 (1962) 146--155.  
%%
%
%%\bibitem{sternin}
%%A.~Savin and B.~Sternin,
%%Pseudodifferential Subspaces and Their Applications in Elliptic Theory,
%%in: {\it $C^*$-algebras and Elliptic Theory}, B.~Bojarski, A.~S.~Mishchenko, E.~V.~Troitsky and A.~Weber (eds), Trends in Mathematics, Birkh\"auser Basel (2006) 247--289. 

%\bibitem{schmudgen}
%K.~Schm\"udgen,
%{\it Unbounded self-adjoint operators on Hilbert space},
%Graduate Texts in Mathematics {\bf 265}, Springer Netherlands, 2012.
%
%\bibitem{Schoen and Yau}
%R.~Schoen and S.-T.~Yau,
%{\it Lectures on differential geometry},
%International Press, 2010.

\bibitem{seeley}
R.T.~Seeley, 
Complex powers of an elliptic operator, 
In: {\it Proc. Symp. Pure Math.} \textbf{10},
Amer. Math. Soc., Providence (RI), 1967, 288--307.
DOI: \href{https://doi.org/10.1090/pspum\%2F010\%2F0237943}{10.1090/pspum\%2F010\%2F0237943}.

%%\bibitem{sharafutdinov}
%%V.~Sharafutdinov,
%%Geometric symbol calculus for pseudodifferential operators: I,
%%{\it Siberian Adv.~Math.} \textbf{15} no.~3 (2005) 81--125. 

\bibitem{shubin}
M.A.~Shubin,
{\it Pseudodifferential operators and spectral theory},
Springer, 2001.
DOI: \href{https://doi.org/10.1007/978-3-642-56579-3}{10.1007/978-3-642-56579-3}.

%\bibitem{singerICM}
%I.M.~Singer, 
%Eigenvalues of the Laplacian and invariants of manifolds,
%{\it Proceedings of the International Congress of Mathematicians} {\bf 1} (1974) 187--200. 
%
%\bibitem{tang and cohen}
%Y.~Tang and A.E.~Cohen,
%Optical chirality and its interaction with matter,
%\emph{Phys. Rev. Lett.} {\bf 104} (2010) 163901.

%%\bibitem{Stiefel}
%%E.~Stiefel,
%%Richtungsfelder und Fernparallelismus in $n$-dimensionalen
%%Mannig\-faltigkeiten.
%%{\it Comment.~Math.~Helv.} \textbf{8} (1935--1936)
%%305--353.
%%
%%\bibitem{Sulanake} 
%%S.~Sulanke, 
%%{\it Berechnung des Spektrums des Quadrates des Dirac-Operators auf der Sph\"are und Untersuchungen zum ersten Eigenwert von D auf 5-dimensionalen R\"aumen konstanter positiver Schnittkr\"ummung}, PhD thesis, Humboldt Universit\"at zu Berlin (1981).
%% 
%
%
%\bibitem{tanno}
%S.~Tanno,
%The first eigenvalue of the Laplacian on spheres,
%\emph{Tohoku Math. J.} {\bf 31} (1979) 179--185.

%\bibitem{taylor_diag}
%M.~E.~Taylor,
%Reflection of singularities of solutions of systems of differential equations,
%{\it Comm. Pure Appl. Math.} {\bf 28} (1975) 457--478.
%
%%\bibitem{Trautman}
%%A.~Trautman, 
%%The Dirac operator on hypersurfaces, {\it Acta Phys.~Pol.~B} \textbf{26} (1995) 1283--1310.
%%
%
%\bibitem{dima_old}
%D.~Vassiliev,
%Two-term asymptotics of the spectrum of a boundary value problem in the case of a piecewise smooth boundary, 
%{\it Sov. Math. Dokl.} {\bf 33} n.~1 (1986) 227--230.
%
%\bibitem{warner}
%F.W.~Warner,
%{\it Foundations of differentiable manifolds and Lie groups},
%Springer, 1983.

\bibitem{wodzicki}
M.~Wodzicki,
Local invariants of spectral asymmetry,
{\it Invent. Math.} {\bf 75} (1984) 143--177.
DOI: \href{https://doi.org/10.1007/BF01403095}{10.1007/BF01403095}.

%
%%\bibitem{woj1}
%%K.~P.~Wojciechowski,
%%A note on the space of pseudodifferential projections with the same principal symbol,
%%{\it J. Operator Theory} \textbf{15} no.~2 (1986) 207--216. 
%%
%%\bibitem{woj2}
%%K.~P.~Wojciechowski,
%%On the Calder\'on projections and spectral projections of the elliptic operators,
%%{\it J. Operator Theory} \textbf{20} no.~1 (1988) 107--115. 

\bibitem{giga}
Z.~Yoshida and Y.~Giga, 
Remarks on spectra of operator rot,
{\it Math. Z.} {\bf 204} (1990) 235--245.
DOI: \href{https://doi.org/10.1007/BF02570870}{10.1007/BF02570870}.

\end{thebibliography}
\end{document}